\documentclass[leqno,10pt,a4paper]{amsart}
\usepackage{etex}
\usepackage{graphicx}
\usepackage{cite}
 \newcommand{\preamble}{
   \keywords{Fano manifolds, mirror symmetry, quantum differential equations, Picard--Fuchs equations, quiver flag varieties}
   \email{e.kalashnikov14@imperial.ac.uk}
   \maketitle
 }

 \usepackage[margin=3cm]{geometry}
\usepackage{xcolor}

\usepackage{hyperref}
\hypersetup{
  colorlinks,
  linkcolor={red!50!black},
  citecolor={blue!50!black},
  urlcolor={blue!80!black}
}

\usepackage[english]{babel}
\usepackage{amsmath}
\usepackage{amssymb}
\usepackage{amsthm}
\usepackage{xypic}
\usepackage{longtable}
\usepackage[vcentermath]{genyoungtabtikz}
\Yboxdim{6pt}
 
\usepackage{epsfig}
\usepackage{hyperref}
\usepackage{enumitem}
\usepackage{booktabs}
 \usepackage{longtable} 
 \usepackage{pdflscape} 
 \usepackage{colortbl} 
 \usepackage{arydshln} 
 \usepackage{calc} 

\newsavebox{\cimatrixbox}
\newlength{\cimatrixheight}
\newlength{\cimatrixshift}

\newcommand{\cimatrix}[1]{%
  \tiny
    \begin{tabular}{c}%
    \savebox{\cimatrixbox}{#1}%
    \setlength{\cimatrixheight}{\totalheightof{\usebox{\cimatrixbox}}}%
    \setlength{\cimatrixshift}{0.5\cimatrixheight}%
    \addtolength{\cimatrixshift}{-0.5ex}%
       \raisebox{\cimatrixshift}{\usebox{\cimatrixbox}}%
     \addtolength{\cimatrixheight}{2.9ex}%
     \rule[-1.45ex]{0pt}{\cimatrixheight}%
     \end{tabular}%
 }
\newcommand{\evnrow}{\rowcolor[gray]{0.95}}
\newcommand{\oddrow}{}

\newcommand{\Obro}[2]{\text{\rm B\O S}^{#1}_{#2}} 
\newcommand{\MM}[2]{\mathrm{MM}^3_{#1\text{--}#2}} 
\newcommand{\BB}[2]{B^{#1}_{#2}}
\newcommand{\VV}[2]{V^{#1}_{#2}}
\newcommand{\MW}[2]{\mathrm{MW}^{#1}_{#2}}
\renewcommand{\SS}[2]{S^2_{#2}}
\newcommand{\FI}[2]{\mathrm{FI}^{#1}_{#2}}
\DeclareMathOperator{\CKP}{CKP}
\DeclareMathOperator{\Str}{Str}

\newtheorem{thm}{\bf Theorem}[section]
\newtheorem{eg}[thm]{\bf Example}
\newtheorem{prop}[thm]{\bf Proposition}
\newtheorem{cor}[thm]{\bf Corollary}
\newtheorem{rem}[thm]{\bf Remark}
\newtheorem{mydef}[thm]{\bf Definition}

\newcommand{\C}{\mathbb{C}}
\newcommand{\R}{\mathbb{R}}
\newcommand{\Z}{\mathbb{Z}}
\newcommand{\Q}{\mathbb{Q}}
\newcommand{\N}{\mathbb{N}}
\newcommand{\PP}{\mathbb{P}}
\newcommand{\Cstar}{\C^*}
\newcommand{\br}{\mathbf{r}}
\newcommand{\be}{\mathbf{e}}

\newcommand{\one}{\mathbf{1}}
\newcommand{\numberofnewFanos}{$141$}

\DeclareMathOperator{\Rep}{\mathrm{Rep}}
\DeclareMathOperator{\Spec}{\mathrm{Spec}}
\DeclareMathOperator{\Hom}{\mathrm{Hom}}
\DeclareMathOperator{\GL}{\mathrm{GL}}
\DeclareMathOperator{\Fl}{\mathrm{Fl}}
\DeclareMathOperator{\OG}{\mathrm{OGr}}
\DeclareMathOperator{\Gr}{\mathrm{Gr}}
\DeclareMathOperator{\Sym}{\mathrm{Sym}}

\DeclareMathOperator{\rk}{\mathrm{rank}}
\DeclareMathOperator{\Cox}{\mathrm{Cox}}
\DeclareMathOperator{\NE}{\mathrm{NE}}

\DeclareMathOperator{\Pic}{\mathrm{Pic}}
\DeclareMathOperator{\Amp}{\mathrm{Amp}}
\DeclareMathOperator{\Nef}{\mathrm{Nef}}
\DeclareMathOperator{\Wall}{\mathrm{Wall}}
\DeclareMathOperator{\Irr}{\mathrm{Irr}}
\DeclareMathOperator{\image}{\mathrm{Im}}
\newcommand{\ab}{\text{\rm ab}}
\renewcommand{\emptyset}{\varnothing}
\newcommand{\intrinsic}[1]{\texttt{\upshape#1}}

\begin{document}

\title{Four-dimensional Fano Quiver Flag Zero Loci}

\author{Elana Kalashnikov}

\address{Department of Mathematics, Imperial College London, 180 Queen's Gate, London SW7 2AZ, UK}

\preamble

\begin{abstract}
  Quiver flag zero loci are subvarieties of quiver flag varieties cut out by sections of homogeneous vector bundles.  We prove the Abelian/non-Abelian Correspondence in this context: this allows us to compute genus zero Gromov--Witten invariants of quiver flag zero loci.  We determine the ample cone of a quiver flag variety, disproving a conjecture of Craw.    In the Appendices, which are joint work with Tom Coates and Alexander Kasprzyk, we use these results to find four-dimensional Fano manifolds that occur as quiver flag zero loci in ambient spaces of dimension up to 8, and compute their quantum periods.  In this way we find at least \numberofnewFanos ~new four-dimensional Fano manifolds.
\end{abstract}

\maketitle

\section{Introduction}

In this paper, we consider quiver flag zero loci: smooth projective algebraic varieties which occur as zero loci of sections of homogeneous vector bundles on quiver flag varieties.  We prove the Abelian/non-Abelian Correspondence of Ciocan-Fontanine--Kim--Sabbah in this context, which allows us to compute genus zero Gromov--Witten invariants of quiver flag zero loci\footnote{Another proof of this, using different methods, has recently been given by Rachel Webb~\cite{Webb2018}.}.  We also determine the ample cone of a quiver flag variety, disproving a conjecture of Craw.  Our primary motivation for these results is as follows.  There has been much recent interest in the possibility of classifying Fano manifolds using Mirror Symmetry.  It is conjectured that, under Mirror Symmetry, $n$-dimensional Fano manifolds should correspond to certain very special Laurent polynomials in $n$~variables~\cite{CoatesCortiGalkinGolyshevKasprzyk2013}.  This conjecture has been established in dimensions up to three~\cite{CoatesCortiGalkinKasprzyk2016}, where the classification of Fano manifolds is  known~\cite{Iskovskih1977,Iskovskih1978,Iskovskih1979,MoriMukai1982,MoriMukai1983,MoriMukai1986,MoriMukai2003,MoriMukai2004}.  Little is known about the classification of four-dimensional Fano manifolds, but there is strong evidence that the conjecture holds for four-dimensional toric complete intersections~\cite{CoatesKasprzykPrince2015}.  Not every Fano manifold is a toric complete intersection, but the constructions in~\cite{CoatesCortiGalkinKasprzyk2016} show that every Fano manifold of dimension at most three is either a toric complete intersection or a quiver flag zero locus.  One might hope, therefore, that most four-dimensional Fano manifolds are either toric complete intersections or quiver flag zero loci.

In the Appendices, which are joint work with Tom Coates and Alexander Kasprzyk, we use the structure theory developed here to find four-dimensional Fano manifolds that occur as quiver flag zero loci in ambient spaces of dimension up to 8, and compute their quantum periods. \numberofnewFanos ~of these quantum periods were previously unknown.  Thus we find at least \numberofnewFanos ~new four-dimensional Fano manifolds. This computation is described in Appendix~\ref{sec:search}.  The quantum periods, and quiver flag zero loci that give rise to them, are recorded in Appendix~\ref{results}.   Figure~\ref{eulerdegree} overleaf shows the distribution of degree and Euler number for the four-dimensional quiver flag zero loci that we found, and for four-dimensional Fano toric complete intersections.

\begin{landscape}
  \begin{figure}[htbp]
    \centering
    \includegraphics[width=0.78\textheight]{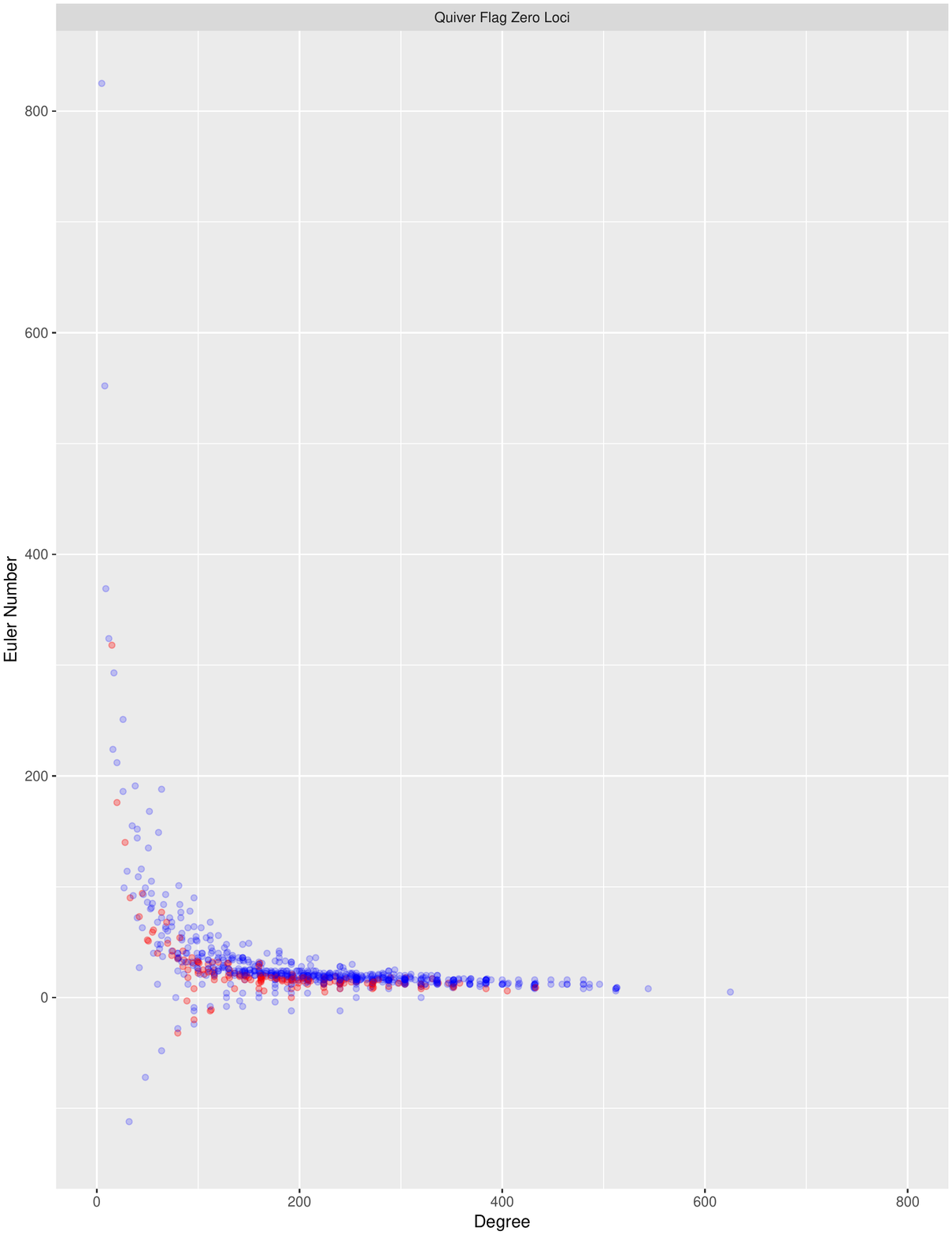}
    \includegraphics[width=0.78\textheight]{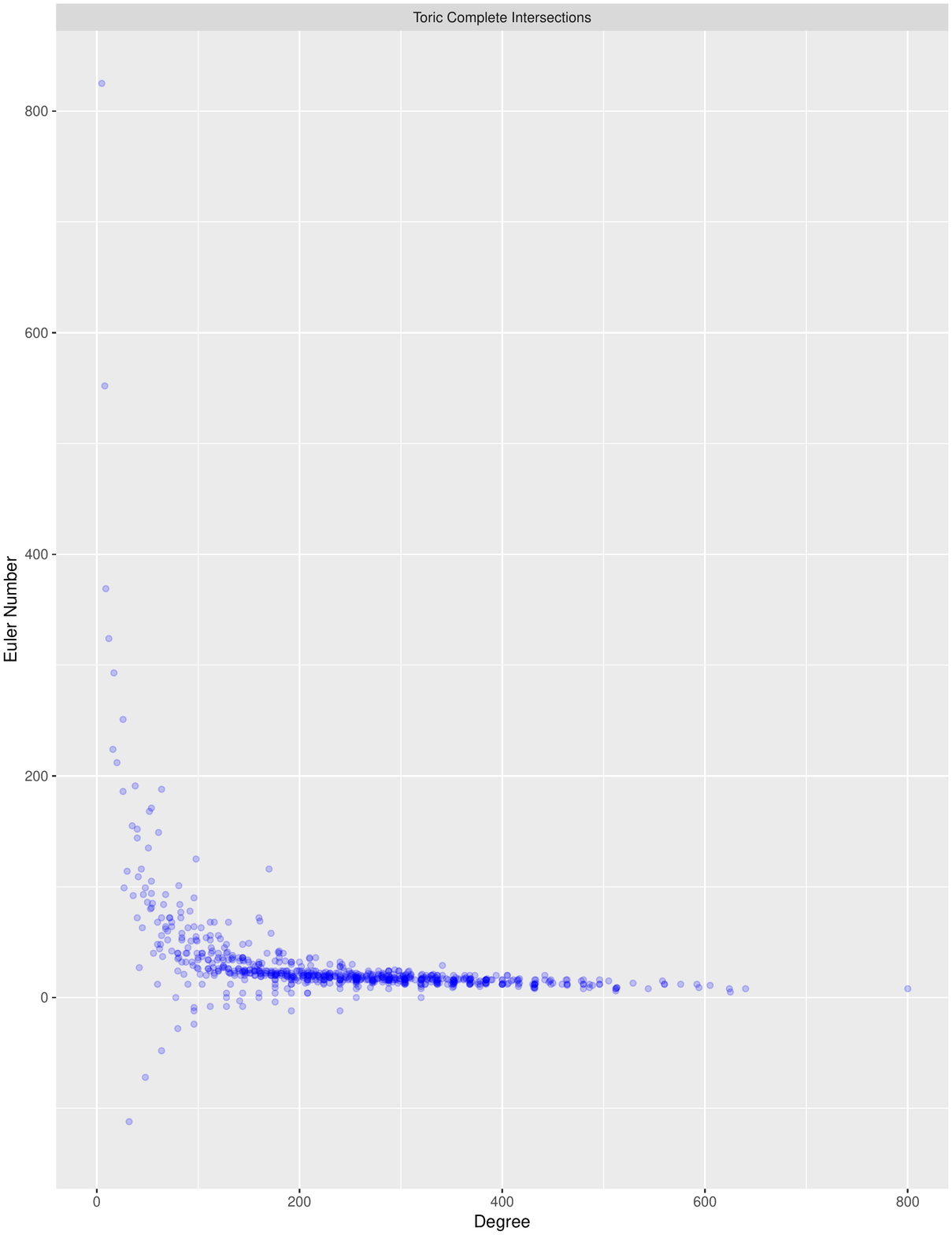}
    \caption{Degrees and Euler numbers for four-dimensional Fano quiver flag zero loci and toric complete intersections; cf.~\cite[Figure~5]{CoatesKasprzykPrince2015}.  Quiver flag zero loci that are not toric complete intersections are highlighted in red.} \label{eulerdegree}
  \end{figure}
\end{landscape}

\section{Quiver flag varieties}
\label{sec:quiver flag varieties}

Quiver flag varieties, which were introduced by Craw~\cite{Craw2011}, are generalizations of Grassmannians and type A flag varieties.  Like flag varieties, they are GIT quotients and fine moduli spaces.  We begin by recalling Craw's construction.  A quiver flag variety $M(Q,\br)$ is determined by a quiver $Q$ and a dimension vector $\br$.  The quiver $Q$ is assumed to be finite and acyclic, with a unique source.  Let $Q_0 = \{0,1,\ldots,\rho\}$ denote the set of vertices of $Q$ and let $Q_1$ denote the set of arrows.  Without loss of generality, after reordering the vertices if necessary, we may assume that $0 \in Q_0$ is the unique source and that the number $n_{ij}$ of arrows from vertex $i$ to vertex $j$ is zero unless $i<j$.  Write $s$,~$t :  Q_1 \to Q_0$ for the source and target maps, so that an arrow $a \in Q_1$ goes from $s(a)$ to $t(a)$.  The dimension vector $\br = (r_0,\ldots,r_\rho)$ lies in $\N^{\rho+1}$, and we insist that $r_0 = 1$. $M(Q,\br)$ is defined to be the moduli space of $\theta$-stable representations of  the quiver $Q$ with dimension vector $\br$. 

\subsection{Quiver flag varieties as GIT quotients.}
Consider the vector space
\[
\Rep(Q,\br) =\bigoplus_{a \in Q_1}\Hom(\C^{r_{s(a)}},\C^{r_{t(a)}})
\]
and the action of $\GL(\mathbf{r}):=\prod_{i=0}^\rho \GL(r_i)$ on $\Rep(Q,\br)$ by change of basis.  The diagonal copy of $\GL(1)$ in $\GL(\br)$ acts trivially, but the quotient $G := \GL(\br)/\GL(1)$ acts effectively; since $r_0 = 1$, we may identify $G$ with $\prod_{i=1}^\rho \GL(r_i)$.  The quiver flag variety $M(Q,\br)$ is the GIT quotient $\Rep(Q,\br)/\!\!/_\theta\, G$, where the stability condition $\theta$ is the character of $G$ given by 
\begin{align*}
  \theta(g) = \prod_{i=1}^\rho \det(g_i), && g = (g_1,\ldots,g_\rho) \in \prod_{i=1}^\rho \GL(r_i).
\end{align*}
 For the stability condition $\theta$, all semistable points are stable.
 To identify the $\theta$-stable points in $\Rep(Q,\br)$, set $s_i=\sum_{a \in Q_1,t(a)=i} r_{s(a)}$ and write
\[
\Rep(Q,\br)=\bigoplus_{i=1}^\rho \Hom(\C^{s_i},\C^{r_i}).
\]
Then $w=(w_i)_{i=1}^\rho$ is $\theta$-stable if and only if $w_i$ is surjective for all $i$. 
\begin{eg} \label{eg:Grassmannian}
Consider the quiver $Q$ given by 
\begin{center}
  \includegraphics[scale=0.5]{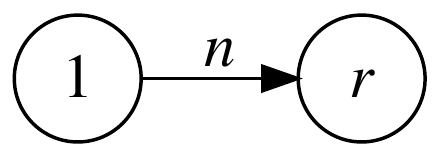}
\end{center}
so that $\rho = 1$, $n_{01} = n$, and the dimension vector $\br = (1, r)$.  Then $\Rep(Q,\br)=\Hom(\C^n,\C^r)$, and the $\theta$-stable points are surjections $\C^n \rightarrow \C^r$. The group $G$ acts by change of basis, and therefore $M(Q,\br) = \Gr(n,r)$, the Grassmannian of $r$-dimensional quotients of $\C^n$.  More generally, the quiver 
\begin{center}
  \includegraphics[scale=0.5]{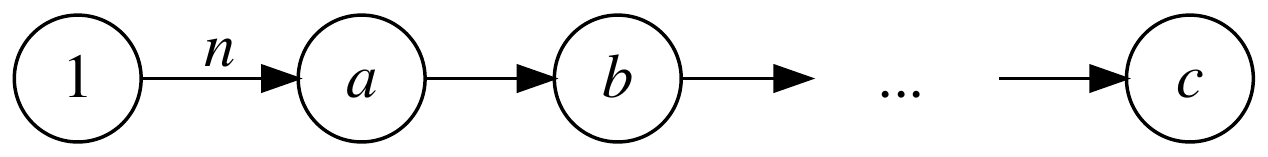}  
\end{center}
gives the flag of quotients $\Fl(n,a,b,\dots,c)$. 
\end{eg}

Quiver flag varieties are non-Abelian GIT quotients unless the dimension vector $\br = (1,1,\ldots,1)$.  In this case $G \cong \prod_{i=1}^\rho \GL_1(\C)$ is Abelian, and $M(Q;\br)$ is a toric variety.  We call such $M(Q,\br)$ toric quiver flag varieties.

\subsection{Quiver flag varieties as moduli spaces.}
\label{sec: moduli spaces}

To give a morphism to $M(Q,\br)$ from a scheme $B$ is the same as to give:
\begin{itemize}[topsep=0.1ex, itemsep=0pt,parsep=0pt]
\item globally generated vector bundles $W_i \to B$, $i \in Q_0$, of rank $r_i$ such that $W_0 = \mathcal{O}_B$; and
\item morphisms $W_{s(a)} \to W_{t(a)}$, $a \in Q_1$ satisfying the $\theta$-stability condition
\end{itemize}
up to isomorphism.  Thus $M(Q,\br)$ carries universal bundles $W_i$, $i \in Q_0$. It is also a Mori Dream Space. The GIT description gives an isomorphism between the Picard group of $M(Q,\br)$ and the character group $\chi(G) \cong \Z^\rho$ of $G$.  When tensored with $\Q$, the fact that this is a Mori Dream space implies that this isomorphism induces an isomorphism of wall and chamber structures given by the Mori structure (on the effective cone) and the GIT structure (on $\chi(G) \otimes \Q$); in particular, the GIT chamber containing $\theta$ is the ample cone of $M(Q,\br)$. The Picard group is generated by the determinant line bundles $\det(W_i)$, $i \in Q_0$. 

\subsection{Quiver flag varieties as towers of Grassmannian bundles.}
\label{sec: tower of Grassmannian bundles}

Generalizing Example~\ref{eg:Grassmannian}, all quiver flag varieties are towers of Grassmannian bundles~\cite[Theorem~3.3]{Craw2011}.
 For $0 \leq i \leq \rho$, let $Q(i)$ be the subquiver of $Q$ obtained by removing the vertices $j \in Q_0$, $j>i$, and all arrows attached to them. Let $\br(i)=(1,r_1, \dots,r_i)$,  and write $Y_i=M(Q(i),\br(i))$.  Denote the universal bundle $W_j \to Y_i$ by $W_j^{(i)}$.
Then there are maps
\[
M(Q,\br)=Y_\rho \to Y_{\rho-1} \to \cdots \to Y_1 \to Y_0=\Spec \C,
\]
induced by isomorphisms $Y_i \cong \Gr(\mathcal{F}_i,r_i)$, where $\mathcal{F}_i$ is the locally free sheaf
\[
\mathcal{F}_i=\bigoplus_{a \in Q_1,t(a)=i}W_{s(a)}^{(i-1)}
\]
of rank $s_i$ on $Y_{i-1}$.  This makes clear that  $M(Q,\br)$ is a smooth projective variety of dimension $\sum_{i=1}^\rho r_i(s_i-r_i)$, and that $W_i$ is the pullback to $Y_\rho$ of the tautological quotient bundle  over $\Gr(\mathcal{F}_i,r_i)$. Thus $W_i$ is globally generated, and $\det(W_i)$ is nef.  Furthermore the anticanonical line bundle of $M(Q,\br)$ is 
\[
\bigotimes_{a \in Q_1}\text{det}(W_{t(a)})^{r_{s(a)}} \otimes \text{det}(W_{s(a)})^{-r_{t(a)}}.
\]
In particular, $M(Q,\br)$ is Fano if $s_i >s_i':=\sum_{a \in Q_1,s(a)=i} r_{t(a)}.$  This condition is not if and only if.


\subsection{Quiver flag zero loci}  

We have expressed the quiver flag variety $M(Q,\br)$ as the quotient by $G$ of the semistable locus $\Rep(Q,\br)^{ss} \subset \Rep(Q,\br)$.  A representation $E$ of $G$, therefore, defines a vector bundle $E_G \to M(Q,\br)$ with fiber $E$; here $E_G = E \times_G \Rep(Q,\br)^{ss}$.  In the second half of this paper, we will study subvarieties of quiver flag varieties cut out by regular sections of such bundles.  We refer to such subvarieties as \emph{quiver flag zero loci}, and such bundles as homogeneous bundles. Quiver flag varieties are also natural ambient spaces via the multi-graded Pl\"ucker embedding (\cite{Craw2011}, \cite{Craw2018}). 

The representation theory of $G = \prod_{i=1}^\rho \GL(r_i)$ is well-understood, and we can use this to write down the bundles $E_G$ explicitly.  Irreducible polynomial representations of $\GL(r)$ are indexed by partitions (or Young diagrams) of length at most $r$.  The irreducible representation corresponding to a partition $\alpha$ is the Schur power $S^\alpha \C^r$ of the standard representation of $\GL(r)$~\cite[chapter~8]{Fulton1997}.  For example, if $\alpha$ is the partition $(k)$ then $S^\alpha \C^r = \Sym^k \C^r$, the $k$th symmetric power, and if $\alpha$ is the partition $(1,1,\ldots,1)$ of length $k$ then $S^\alpha \C^r = \bigwedge^k \C^r$, the $k$th exterior power.  Irreducible polynomial representations of $G$ are therefore indexed by tuples $(\alpha_1,\ldots,\alpha_\rho)$ of partitions, where $\alpha_i$ has length at most $r_i$. The tautological bundles on a quiver flag variety are homogenous: if $E=\C^{r_i}$ is the standard representation of the $i^{th}$ factor of $G$, then $W_i=E_G$. More generally, the representation indexed by $(\alpha_1,\ldots,\alpha_\rho)$ is $\bigotimes_{i=1}^\rho S^{\alpha_i} \C^{r_i}$, and the corresponding vector bundle on $M(Q,\br)$ is $\bigotimes_{i=1}^\rho S^{\alpha_i} W_i$. 

\begin{eg} Consider the vector bundle $\Sym^2 W_1$ on $\Gr(8,3)$. By duality -- which sends a quotient $\C^8 \to V \to 0$ to a subspace $0 \to V^* \to (\C^8)^*$ --  this is equivalent to considering the vector bundle $\Sym^2 S_1^*$ on the Grassmannian of 3-dimensional subspaces of $(\C^8)^*$, where $S_1$ is the tautological sub-bundle. A generic symmetric 2-form $\omega$ on $(\C^8)^*$ determines a regular section of $\Sym^2 S_1^*$, which vanishes at a point $V^*$ if and only if the restriction of $\omega$ to $V^*$ is identically zero. So the associated quiver flag zero locus is the orthogonal Grassmannian  $\OG(3,8)$.
\end{eg}

\subsection{The Euler sequence}

Quiver flag varieties, like both Grassmannians and toric varieties, have an Euler sequence.

\begin{prop} \label{euler}
  Let $X = M(Q,\br)$ be a quiver flag variety, and for $a \in Q_1$, denote $W_a:=W_{s(a)}^* \otimes W_{t(a)}.$ There is a short exact sequence
  \[
    0 \to \bigoplus_{i=1}^\rho W_i \otimes W_i^* \to \bigoplus_{a \in Q_1} W_a \to T_X \to 0.
  \]
\end{prop} 

\begin{proof}
  We proceed by induction on the Picard rank $\rho$ of $X$.  If $\rho=1$ then this is the usual Euler sequence for the Grassmannian.
  Suppose that the proposition holds for quiver flag varieties of Picard rank $\rho-1$, $\rho>1$. Then the fibration $\pi \colon \Gr(\pi^*\mathcal{F}_\rho, r_\rho) \to Y_{\rho-1}$ from \S\ref*{sec:quiver flag varieties}\ref{sec: tower of Grassmannian bundles} above gives a short exact sequence
\[
0 \to W_\rho \otimes W_\rho^* \to \pi^*\mathcal{F}_\rho^* \otimes W_\rho \to S^* \otimes W_\rho \to 0
\]
where $S$ is the kernel of the projection $\pi^* \mathcal{F}_\rho \to W_\rho$. Note that
\begin{align*}
 \pi^*\mathcal{F}_\rho^* \otimes W_\rho=\bigoplus_{a \in Q_1, t(a)=\rho} W_a
  && \text{and that} && 
  T_X=T_{Y_{\rho-1}} \oplus S^* \otimes W_\rho.
\end{align*}
As $S^* \otimes W_\rho$ is the relative tangent bundle to $\pi$, the proposition follows by induction.
\end{proof}

If $X$ is a quiver flag zero locus cut out of the quiver flag variety $M(Q,\br)$ by a regular section of the homogeneous vector bundle $E$ then there is an exact sequence
\[
0 \to T_X \to T_{M(Q,\br)}|_X \to E \to 0.
\]
Thus $T_X$ is the K-theoretic difference of homogeneous vector bundles.

\section{Quiver flag varieties as subvarieties}\label{zero loci}
There are three well-known constructions of flag varieties: as GIT quotients, as homogenous spaces, and as subvarieties of products of Grassmannians. Craw's construction gives quiver flag varieties as GIT quotients. General quiver flag varieties are not homogenous spaces, so the second construction does not generalize. In this section we generalize the third construction of flag varieties, exhibiting quiver flag varieties as subvarieties of products of Grassmannians.  It is this description that will allow us to prove the Abelian/non-Abelian correspondence for quiver flag varieties.

\begin{prop} \label{prop:zl}
Let $M(Q,\br)$ be a quiver flag variety with $\rho>1.$ Let $\tilde{s_i}= \dim H^0(M(Q,\br),W_i)$. Then $M(Q,\br)$ is cut out of $Y=\prod_{i=1}^\rho \Gr(\tilde{s}_i,r_i)$ by a canonical section of 
\[
E=\bigoplus_{a \in Q_1, s(a) \neq 0} S_{s(a)}^* \otimes Q_{t(a)}
\]
where $S_i$ and $Q_i$ are the tautological sub-bundle and quotient bundle on the $i^{th}$ factor of $Y$.
\end{prop}
\begin{proof}
Recall that $\tilde{s}_i$ is the dimension of the space of paths from 0 to vertex $i$ (Corollary 3.5, \cite{Craw2011}). Thus 
\[
\C^{\tilde{s}_i}=\bigoplus_{a \in Q_1, t(a)=i, s(a) \neq 0} \C^{\tilde{s}_{s(a)}} \oplus \C^{n_{0i}}.
\]
Let $F_i =\bigoplus_{t(a)=i} Q_{s(a)}$. Combining the canonical surjections $\C^{\tilde{s}_{s(a)}} \otimes \mathcal{O} \to Q_{s(a)}$  gives a surjection $\C^{\tilde{s}_i} \otimes \mathcal{O} \rightarrow F_i$
that fits into the exact sequence
\[
0 \to \bigoplus_{t(a)=i, s(a) \neq 0} S_{s(a)} \to \C^{\tilde{s_i}}\otimes \mathcal{O} \to F_i \to 0.
\]
Thus 
\[
(\C^{\tilde{s}_i}{}^* \otimes \mathcal{O})/F_i^* \cong \bigoplus_{t(a)=i,s(a) \neq 0} S^*_{s(a)}
\]
and it follows that $E=\bigoplus_{i=2}^\rho \Hom(Q_i^*, (\C^{\tilde{s}_i}{}^* \otimes \mathcal{O})/F_i^*)$.

Consider the section $s$ of $E$ given by the compositions $Q_i^* \to \C^{\tilde{s}_i}{}^* \otimes \mathcal{O} \to (\C^{\tilde{s}_i}{}^* \otimes \mathcal{O})/F_i^*.$
The section $s$ vanishes at quotients $(V_1,\dots,V_\rho)$ if and only if $V_i^* \subset \bigoplus_{t(a)=i} V_{s(a)}^*$; dually, the zero locus is where there is a surjection $F_i \to Q_i$ for each $i$.  Now it isn't hard to see that $Z(s)$ is $M(Q,\br)$.  $M(Q,\br)$ parametrizes  surjections $\C^{\tilde{s_i}} \to V_i$ that factor as
\[
\C^{\tilde{s_i}} \to \bigoplus_{t(a)=i} V_{s(a)} \to V_i.
\]
Since the $W_i$ are globally generated, there is a unique map
\[
M(Q,\br) \to Y=\prod_{i=1}^\rho \Gr(\tilde{s_i},r_i)
\]
such that $Q_i \to Y$ pulls back to $W_i \to M(Q,\br)$. The discussion above shows that the image of this map lies in the zero locus of $s$. 
Similarly, the universal property of $M(Q,\br)$ gives rise to a unique map $Z(s) \to M(Q,\br)$ such that $W_i \to M(Q,\br)$ pulls back to $Q_i \to Z(s)$.  Because $M(Q,\br)$ is a fine moduli space, the composition of these maps $M(Q,\br) \to Z(s) \to M(Q,\br)$ must be the identity.  Similarly, the composition $ Z(s) \to M(Q,\br) \to Z(s)$ is the identity, and we conclude that $Z(s)$ and $M(Q,\br)$ are canonically isomorphic.
\end{proof}

Suppose that $X$ is a quiver flag zero locus cut out of $M(Q,\br)$ by a regular section of a homogeneous vector bundle~$E$. Since $M(Q,\br)$ and the product of Grassmannians described above are both GIT quotients by the same group $G$, the representation of $G$ that determines $E$ also determines a vector bundle $E'$ on $\prod_{i=1}^\rho \Gr(\tilde{s_i},r_i)$.  We see that $X$ is deformation equivalent to the zero locus of a generic section of the vector bundle $F := E' \oplus \bigoplus_{a \in Q_1, s(a) \neq 0} S_{s(a)}^* \otimes Q_{t(a)}$. Although the product of Grassmannians is a quiver flag variety, this is not generally an additional model of $X$ as a quiver flag zero locus, as the summand $S_{s(a)}^* \otimes Q_{t(a)}$ in $F$ does not in general come from a representation of $G$. 

\begin{rem}\label{gg} Suppose $\alpha$ is a non-negative Schur partition. Then \cite{Snow86} shows that $S^\alpha(Q)$ is globally generated on $\Gr(n,r)$. The above construction allows us to generalise this to quiver flag varieties: $S^\alpha(W_i)$ is globally generated on $M(Q,\br)$.
\end{rem}

\section{Equivalences of quiver flag zero loci} \label{sec:equivalences}

The representation of a given variety $X$ as a quiver flag zero locus, if it exists, is far from unique.  In this section we describe various methods of passing between different representations of the same quiver flag zero locus.  This is important in practice, because our systematic search for four-dimensional quiver flag zero loci described in the Appendix finds a given variety in many different representations.  Furthermore, geometric invariants of a quiver flag zero locus $X$ can be much easier to compute in some representations than in others.  The observations in this section allow us to compute invariants of four-dimensional Fano quiver flag zero loci using only a few representations, where the computation is relatively cheap, rather than doing the same computation many times and using representations where the computation is expensive.

\subsection{Dualising}

As we saw in the previous section, a quiver flag zero locus $X$ given by $(M(Q,\br), E)$ can be thought of as a zero locus in a product of Grassmannians $Y$. Unlike general quiver flag varieties, Grassmannians come in canonically isomorphic dual pairs:
\begin{center}
    \includegraphics[scale=0.5]{Gr_n_r.pdf}
    \qquad \qquad
    \includegraphics[scale=0.5]{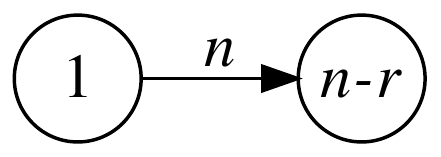}
\end{center}
The isomorphism interchanges the tautological quotient bundle $Q$ with $S^*$, where $S$ is the tautological sub-bundle. One can then dualize some or none of the Grassmannian factors in $Y$, to get different models of $X$. Depending on the representations in $E$, after dualizing, $E$ may still be a homogenous vector bundle, or the direct sum of a homogeneous vector bundle with bundles of the form $S_i^* \otimes W_j$. If this is the case, one can then undo the product representation process to obtain another model $(M(Q', \br'),E'_G)$ of $X$.

\begin{eg} Consider $X$ given by the quiver
\begin{center}
  \includegraphics[scale=0.5]{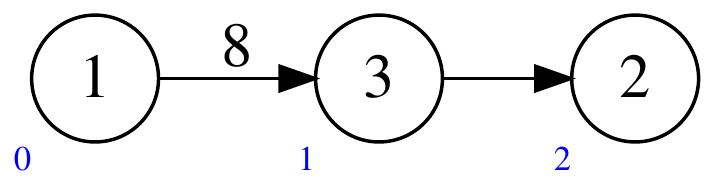} 
\end{center}
and bundle $\wedge^2 W_2$; here and below the vertex numbering is indicated in blue.
Then writing it as a product:
\begin{center}
  \includegraphics[scale=0.5]{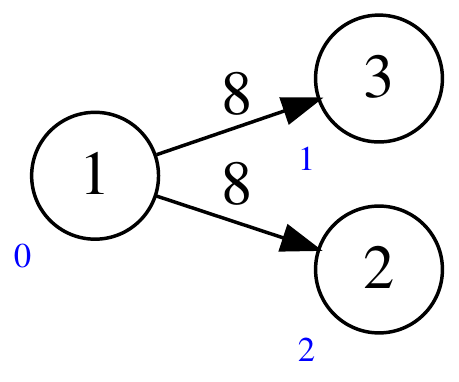} 
\end{center}
with bundle $\wedge^2 W_2 \oplus S_1^* \otimes W_2$ and dualizing the first factor, we get
\begin{center}
  \includegraphics[scale=0.5]{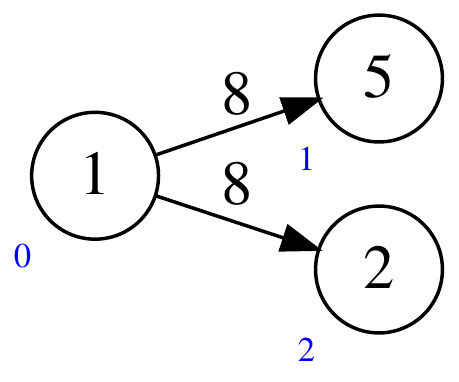} 
\end{center}
with bundle $W_1 \otimes W_2 \oplus \wedge^2 W_2$, which is a quiver flag zero locus. 
\end{eg}

\subsection{Removing arrows}

\begin{eg} \label{eg:Gr arrow}
  Recall that $\Gr(n,r)$ is the quiver flag zero locus given by $(\Gr(n+1,r), W_1)$. This is because the space of sections of $W_1$ is $\C^{n+1}$, where the image of the section corresponding to $v \in \C^{n+1}$ at the point $\phi \colon \C^{n+1} \to  W$ in $\Gr(n+1,r)$ is $\phi(v)$. This section vanishes precisely when $v \in \ker{\phi}$, so we can consider its zero locus to be $\Gr(\C^{n+1}/\langle v\rangle,r) \cong \Gr(n,r)$. The restriction of $W_1$ to this zero locus $\Gr(n,r)$ is $W_1$, and the restriction of the tautological sub-bundle $S$ is $S \oplus \mathcal{O}_{\Gr(n,r)}$.  
\end{eg}

This example generalises.  Let $M(Q,\br)$ be a quiver flag variety.  A choice of arrow $i \to j$ in $Q$ determines a canonical section of $W_i^* \otimes W_j$, and the zero locus of this section is $M(Q',\br)$,  where $Q'$ is the quiver obtained from $Q$ by removing one arrow from $i \to j$.

\begin{eg} \label{eg:Gr det} 
  Similarly, $\Gr(n,r)$ is the zero locus of a section of $S^*$, the dual of the tautological sub-bundle, on $\Gr(n+1,r+1)$. The exact sequence $0 \to W_1^* \to (\C^{n+1})^* \to S^* \to 0$
  shows that a global section of $S^*$ is given by a linear map $\psi: \C^{n+1} \to \C$. The image of the section corresponding to $\psi$ at the point $s \in S$ is $\psi(s)$, where we evaluate $\psi$ on $s$ via the tautological inclusion $S \to \C^{n+1}$. Splitting $\C^{n+1}=\C^{n}\oplus \C$ and choosing $\psi$ to be projection to the second factor shows that $\psi$ vanishes precisely when $S \subset \C^{n}$, that is, precisely along $\Gr(n,r)$.  The restriction of $S$ to this zero locus $\Gr(n,r)$ is $S$, and the restriction of $W_1$ is $W_1 \oplus \mathcal{O}_{\Gr(n,r)}$.
\end{eg}
\subsection{Grafting}

Let $Q$ be a quiver and $S \subset \{1,\dots,\rho\}$. We say that a vertex $i$ is \emph{graftable for $S$} if: 
\begin{itemize}[topsep=0.1ex, itemsep=0pt,parsep=0pt]
\item $r_i=1$ and $0<i<\rho$;
\item there is a path between vertex $i$ and vertex $j$ for all $j \in S$;
\item if we remove all of the arrows from $i$ to vertices in $S$, we get a disconnected quiver. 
\end{itemize}
\begin{eg} 
In the quiver below, vertex $1$ is not graftable for $\{2\}$.
\begin{center}
  \includegraphics[scale=0.5]{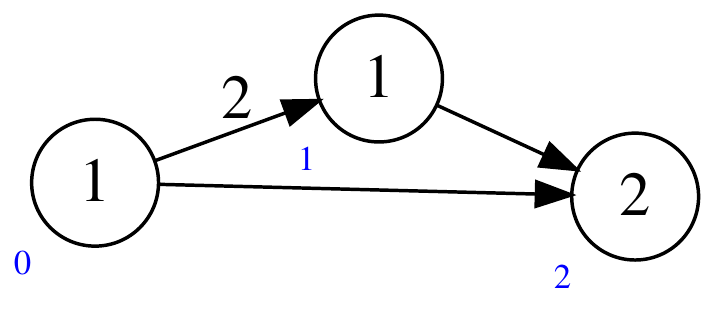} 
\end{center}
If we removed the arrow from vertex $0$ to vertex $2$, then vertex $1$ would be graftable for $\{2\}$. 
\end{eg}

\begin{prop} 
  Let $Q$ be a quiver and let $i$ be a vertex of $Q$ that is graftable for $S$. Let $Q'$ be the quiver obtained from $Q$ by replacing each arrow $i\to j$, where $j \in S$, by an arrow $0 \to j$.  Then 
  \[
   M(Q,\br) = M(Q', \br)
   \]
   for any dimension vector $\br$.
\end{prop}
\begin{proof}
  Recall from \S\ref*{sec:quiver flag varieties}\ref{sec: moduli spaces} the moduli problem represented by $M(Q,\br)$ and note, since $i$ is graftable for $S$, that $W_i$ is a line bundle.  Setting
  \[
    W_j \mapsto
    \begin{cases}
      W_i^* \otimes W_j & j \in S \\
      W_j & j \not \in S
    \end{cases}
  \]
  transforms the moduli problem for $M(Q,\br)$ into that for $M(Q',\br)$.
\end{proof}

\begin{eg} Consider the quiver flag zero locus $X$ given by the quiver in (a) below, with bundle 
  \[
  W_1 \otimes W_3 \oplus W_1^{\oplus 2} \oplus \det W_1.
  \]
  Writing $X$ inside a product of Grassmannians gives $W_1 \otimes W_3 \oplus W_1^{\oplus 2} \oplus \det W_1$ on the quiver in (b), with arrow bundle $S_2^* \otimes W_1$.  Removing the two copies of $W_1$ using Example~\ref{eg:Gr arrow} gives 
  \[
  W_1 \otimes W_3 \oplus \det W_1
  \]
  on the quiver in (c), with arrow bundle $S_2^* \otimes W_1$. Applying Example~\ref{eg:Gr det} to remove $\det W_1 = \det S_1^* = S_1^*$, and taking care to absorb the extra factors of $W_3$ and $S_2^*$ which arise from $W_1 \otimes W_3$ and the arrow bundle, we see that $X$ is given by $W_1 \otimes W_3$ on the quiver in (d), with arrow bundle $S_2^* \otimes W_1$.  Dualising at vertices~$1$ and~$2$ now gives the quiver in (e), with arrow bundle $S_1^* \otimes W_2 \oplus S_1^* \otimes W_3$.  Finally, undoing the product representation exhibits $X$ as the quiver flag variety for the quiver in (f).
  \begin{center}
    (a) \begin{minipage}[b]{0.3\textwidth}\includegraphics[scale=0.5]{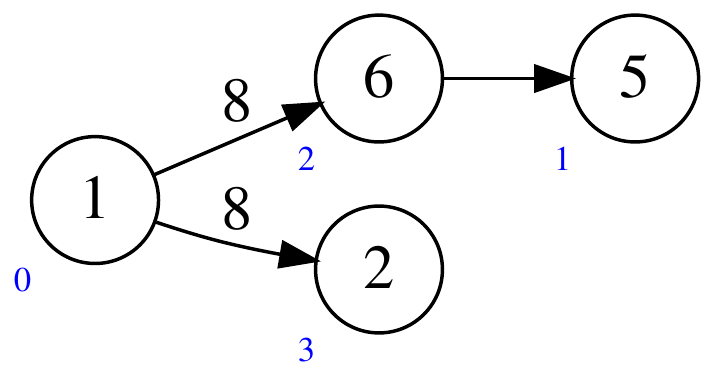}\end{minipage}
    (b) \begin{minipage}[b]{0.25\textwidth}\includegraphics[scale=0.5]{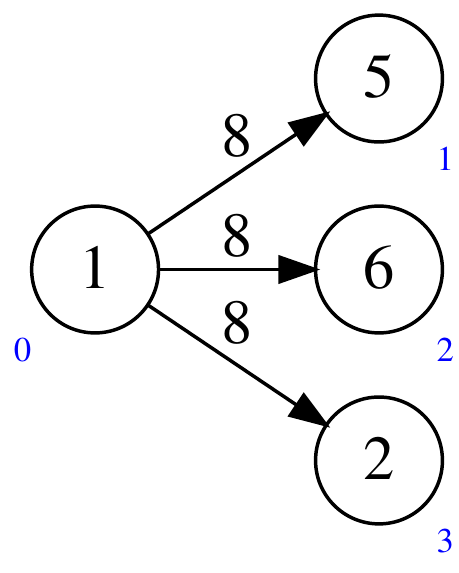}\end{minipage} 
    (c) \begin{minipage}[b]{0.25\textwidth}\includegraphics[scale=0.5]{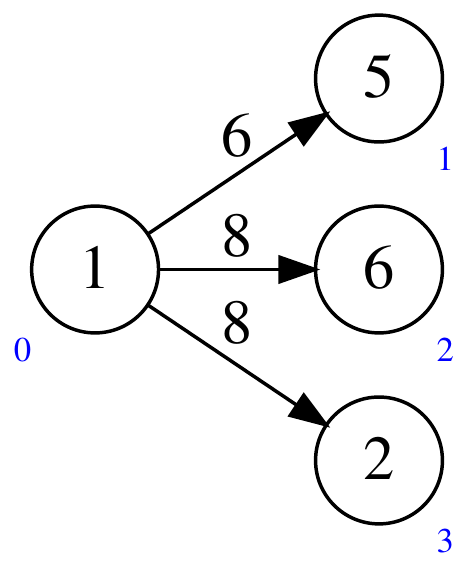}\end{minipage}  \\
    (d) \begin{minipage}[b]{0.3\textwidth}\includegraphics[scale=0.5]{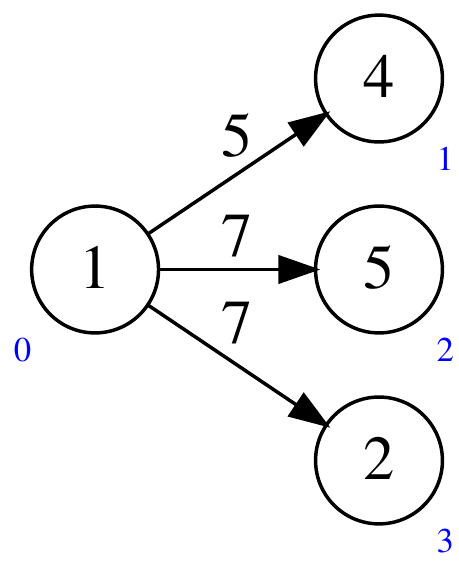}\end{minipage} 
    (e) \begin{minipage}[b]{0.25\textwidth}\includegraphics[scale=0.5]{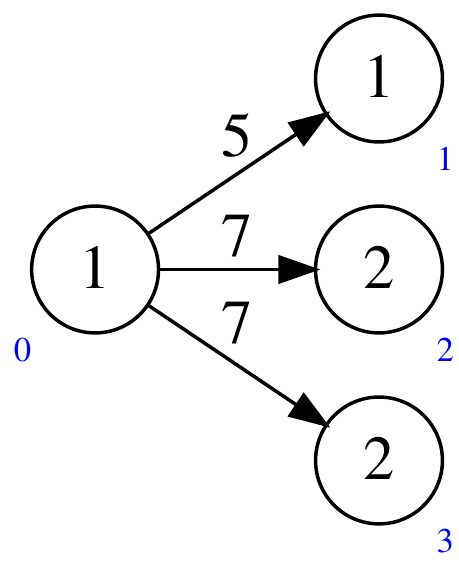}\end{minipage} 
    (f) \begin{minipage}[b]{0.25\textwidth}\includegraphics[scale=0.5]{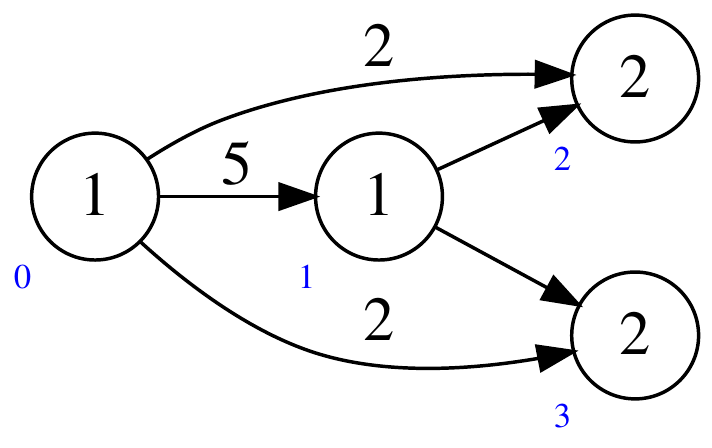}\end{minipage} 
  \end{center}

\end{eg}
\section{The ample cone}

We now discuss how to compute the ample cone of a quiver flag variety.  This is essential if one wants to search systematically for quiver flag zero loci that are Fano.  In~\cite{Craw2011}, Craw gives a conjecture that would in particular solve this problem, by relating a quiver flag variety $M(Q,\br)$ to a toric quiver flag variety. We give a counterexample to this conjecture, and determine the ample cone of $M(Q,\br)$ in terms of the combinatorics of the quiver: this is Theorem~\ref{thm:nef cone} below. Our method also involves a toric quiver flag variety: the Abelianization of $M(Q,\br)$.

\subsection{The multi-graded Pl\"{u}cker embedding}  
Given a quiver flag variety $M(Q,\br)$, Craw defines a multi-graded analogue of the Pl\"{u}cker embedding:
\begin{align*}
  p:M(Q,\br) \hookrightarrow M(Q',\one) && \text{with $\one=(1,\dots,1)$.}
\end{align*}
Here $Q'$ is the quiver with the same vertices as $Q$ but with the number of arrows $i \to j$ given by 
\[
\dim H^0\left(\Hom\big(\det(W_i),\det(W_j)\big)\right)/S
\]
where $S$ is spanned by maps which factor through maps to $\det(W_k)$ with $i<k<j$ (\cite{Craw2011}, Example 2.9 in ~\cite{Craw2018}).  This induces an isomorphism $p^*: \\Pic(X)\otimes \R \to \Pic(X)\otimes \R$ that sends $\det(W_i') \mapsto \det(W_i).$ In~\cite{Craw2011}, it is conjectured that this induces a surjection of Cox rings $\Cox(M(Q',\one)) \to \Cox(M(Q,\br))$. 
This would give information about the Mori wall and chamber structure of $M(Q,\br)$. In particular, by the proof of Theorem 2.8 of~\cite{Levitt2014}, a surjection of Cox rings together with an isomorphism of Picard groups (which we have here) implies an isomorphism of effective cones. 

We provide a counterexample to the conjecture. To do this, we exploit the fact that quiver flag varieties are Mori Dream Spaces, and so the Mori wall and chamber structure on $\NE^1(M(Q,\br) \subset \Pic(M(Q,\br)$ coincides with the GIT wall and chamber structure. This gives GIT characterizations for effective divisors, ample divisors, nef divisors, and the walls. 
 
\begin{thm} \cite{DolgachevHu1998} Let $X$ be a Mori Dream Space obtained as a GIT quotient of $G$ acting on $V=\C^N$ with stability condition $\omega \in \chi(G)=\Hom(G,\C^*)$. Identifying $\Pic(X) \cong \chi(G)$, we have that: 
\begin{itemize}[topsep=0.1ex, itemsep=0pt,parsep=0pt]
\item $v \in \chi(G)$ is ample if $V^{s}(v)=V^{ss}(v)=V^{s}(\omega)$.
\item $v$ is on a wall if $V^{ss}(v) \neq V^{s}(v)$.
\item $v \in \NE^1(X)$ if $V^{ss} \neq \emptyset$.
\end{itemize}
\end{thm}

When combined with King's characterisation~\cite{King1994} of the stable and semistable points for the GIT problem defining $M(Q,\br)$, this determines the ample cone of any given quiver flag variety.  In Theorem~\ref{thm:nef cone} below we make this effective, characterising the ample cone in terms of the combinatorics of $Q$, but this is already enough to see a counterexample to the conjecture.

\begin{eg}
Consider the quiver $Q$ and dimension vector $\br$ as in (a).  The target $M(Q',\one)$ of the multi-graded Pl\"ucker embedding has the quiver $Q'$ shown in (b).
\begin{center}
  (a) \begin{minipage}[t]{0.3\textwidth}\includegraphics[scale=0.5]{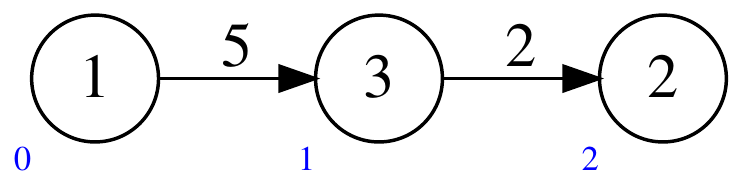}\end{minipage} \qquad \qquad
  (b) \begin{minipage}[b]{0.3\textwidth}\includegraphics[scale=0.5]{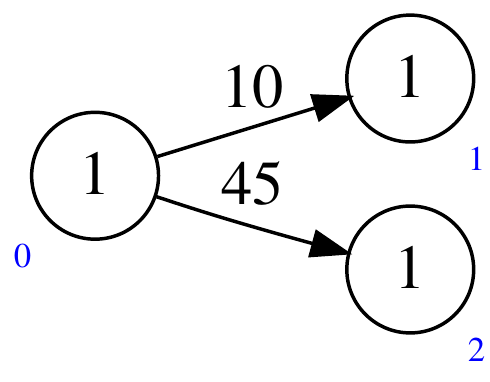}\end{minipage} 
\end{center}
In this case, $M(Q',\one)$ is a product of projective spaces and so the effective cone coincides with the nef cone, which is just the closure of the positive orthant.  The ample cone of $M(Q,\br)$ is indeed the positive orthant, as we will see later.  However, the effective cone is strictly larger. We will use King's characterisation of semi-stable points with respect to a character $\chi$ of $\prod_{i=0}^\rho Gl(r_i)$: a representation $R = (R_i )_{i \in Q_0}$ is semi-stable with respect to $\chi=(\chi_i)_{i=0}^\rho$ if and only if
\begin{itemize}[topsep=0.1ex, itemsep=0pt,parsep=0pt]
\item $\sum_{i=0}^\rho \chi_i \dim_{\C}(R_i)=0$; and
\item for any  subrepresentation $R'$ of $R$, $\sum_{i=0}^\rho \chi_i \dim_{\C}(R'_i)\geq 0$.
\end{itemize}

Consider the character $\chi=(-1,3)$ of $G$, which we lift to a character of $\prod_{i=0}^\rho Gl(r_i)$ by taking $\chi=(-3,-1,3)$. We will show that there exists a representation $R=(R_0,R_1,R_2)$ which is semi-stable with respect to $\chi$. 
The maps in the representation are given by a triple $(A,B,C) \in Mat(3 \times 5) \times Mat(2 \times 3) \times Mat(2 \times 3).$ Suppose that
\begin{align*}
  \text{$A$ has full rank}, &&
  B=\begin{bmatrix}
    1 & 0& 0\\
    0 & 1 &0 \\
  \end{bmatrix}, &&
  C=\begin{bmatrix}
    0 & 0 & 0\\
    0 & 0 & 1 \\
  \end{bmatrix},
\end{align*}
and that $R'$ is a subrepresentation with dimensions $a$,~$b$,~$c$. We want to show that $-3 a - b + 3c \geq 0.$ If $a=1$ then $b=3$, as otherwise the image of $A$ is not contained in $R_1'$. Similarly, this implies that $c=2$. So suppose that $a=0$. The maps $B$ and $C$ have no common kernel, so $b > 0$ implies $c > 0$, and $-b+3 c \geq 0$  as $b \leq 3$.  Therefore $R$ is a semi-stable point for $\chi$, and as quiver flag varieties are Mori Dream Spaces, $\chi$ is in the effective cone. 
\end{eg}

\subsection{Abelianization}

We consider now the toric quiver flag variety associated to a given quiver flag variety $M(Q,\br)$ which arises from the corresponding Abelian quotient.  Let $T \subset G$ be the diagonal maximal torus. Then the action of $G$ on $\Rep(Q,\br)$ induces an action of $T$ on $\Rep(Q,\br)$, and the inclusion $i:\chi(G) \hookrightarrow \chi(T)$ allows us to interpret the special character $\theta$ as a stability condition for the action of $T$ on $\Rep(Q,\br)$. The Abelian quotient is then $\Rep(Q,\br)/\!\!/_\theta T$.
Let us see that $\Rep(Q,\br)/\!\!/_\theta T$ is a toric quiver flag variety. Let $\lambda=(\lambda_1,\dots,\lambda_\rho)$ denote an element of $T=\prod_{i=1}^\rho (\C^*)^{r_i}$, where $\lambda_j=(\lambda_{j1},\dots,\lambda_{j r_j})$. The action of $\lambda$ on the $(i,j)$ entry $x$ of the $r_{t(a)} \times r_{s(a)}$ matrix which is the $a \in Q_1$ component of an element in $\Rep(Q,\br)$
is 
\[
x \mapsto \lambda_{s(a) i}^{-1} x \lambda_{t(a) j}.
\]
Hence this is the same as the group action on the quiver $Q^{\ab}$ with vertices 
\[
Q^{\ab}_0=\{v_{i j} : 0 \leq i \leq \rho, 1 \leq j \leq r_i\}
\]
and the number of arrows between $v_{i j}$ and $v_{i k}$ is the number of arrows in the original quiver between vertices $i$ and $k$. Hence 
\[
\Rep(Q,\br)/\!\!/_\theta T=M(Q^{\ab},\one).
\]
We call $Q^{\ab}$ the \emph{Abelianized} quiver. 
\begin{eg} Let $Q$ be the quiver
\begin{center}
  \includegraphics[scale=0.5]{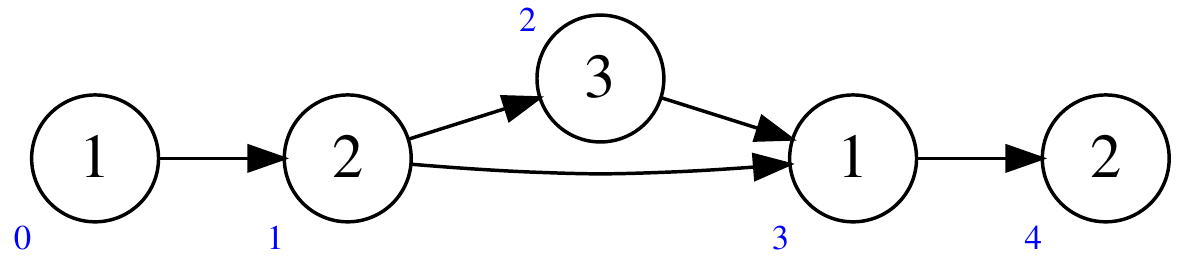} 
\end{center}
Then $Q^{\ab}$ is
\begin{center}
  \includegraphics[scale=0.5]{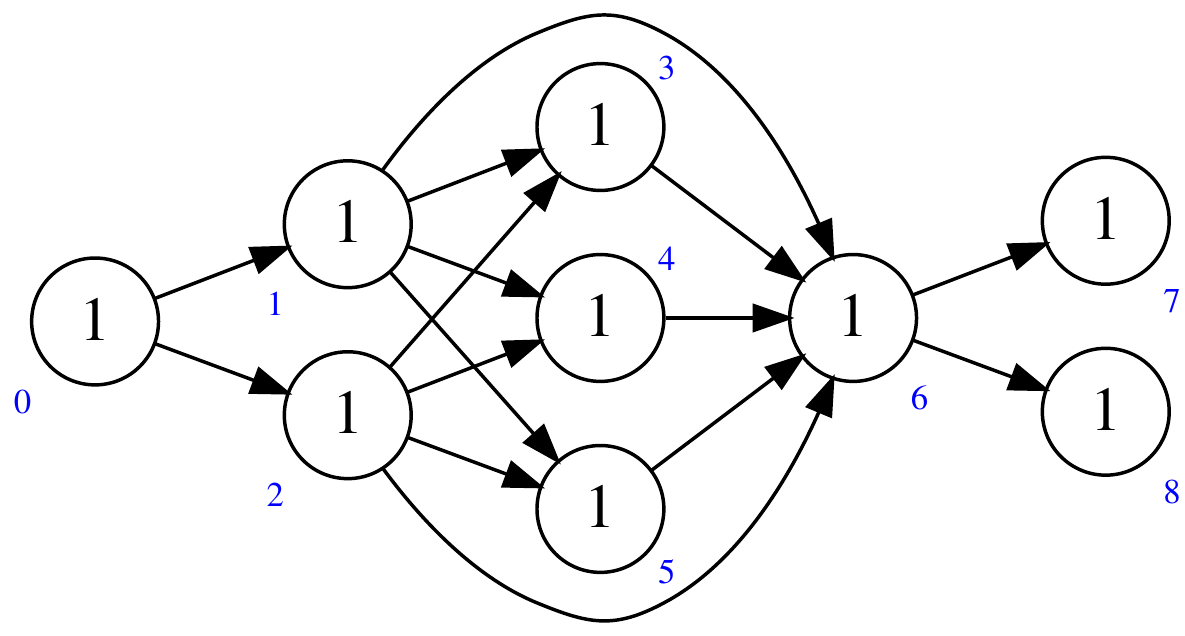} 
\end{center}
\end{eg}

Martin~\cite{Martin2000} has studied the relationship between the cohomology of Abelian and non-Abelian quotients.  We state his result specialized to quiver flag varieties, then extend this to a comparison of the ample cones.  To simplify notation, denote $M_Q=M(Q,\br)$, $M_{Q^{\ab}}=M(Q^{\ab},(1,\dots,1))$ and $V=\Rep(Q,\br)=\Rep(Q^{\ab},(1,\dots,1))$. For $v \in \chi(G)$, let $V_v^s(T)$ denote the $T$-stable points of $V$ and $V_v^s(G)$ denote the $G$-stable points, dropping the subscript if it is clear from context. It is easy to see that $V^s(G) \subset V^s(T)$. The Weyl group $W$ of $(G,T)$ is $\prod_{i=1}^\rho S_{r_i}.$ Let $\pi:V^s(G)/T \to V^s(G)/G$ be the projection.  The Weyl group acts on the cohomology of $M(Q^{\ab},\one)$, and also on the Picard group, by permuting the $W_{v_{i1}},\dots,W_{v_{i r_i}}.$  It is well-known (see e.g.~Atiyah--Bott~\cite{AtiyahBott84}) that 
\[
\pi^*: H^*(V^s(G)/T)^W \cong H^*(M_Q).
\]

\begin{thm}{\cite{Martin2000}}\label{martin} 
  There is a graded surjective ring homomorphism
  \[
  \phi:H^*(M_{Q^{\ab}},\C)^W  \to H^*(V^{s}(G)/T,\C) \xrightarrow{\pi^*} H^*(M_Q,\C)
  \]
  where the first map is given by the restriction $V^{s}(T)/T \to V^{s}(G)/T$.  The kernel is the annihilator of $e=\prod_{i=1}^\rho \prod_{1 \leq j,k \leq r_i} c_1(W_{v_{i j}}^* \otimes W_{v_{i k}}).$
\end{thm} 
\begin{rem} This means that any class $\sigma \in H^*(M_Q)$ can be lifted (non-uniquely) to a class $\tilde{\sigma} \in H^*(M_{Q^{\ab}})$. Moreover, $e \cap \tilde{\sigma}$ is uniquely determined by $\sigma$.
\end{rem}
\begin{cor} Let $E$ be a representation of $G$ and hence $T$. We can form the vector bundles
  \begin{align*}
    E_G&=(V^s(G) \times E)/G \to M_Q,\\
    E_G'&=(V^s(G) \times E)/T \to V^s(G)/T,\\
    E_T&=(V^s(T) \times E)/T \to M_{Q^{\ab}},
  \end{align*}
  where in all cases the group acts diagonally.
Then $\phi(c_i(E_T))=c_i(E_G).$
\end{cor}
\begin{proof}
Let $f$ be the inclusion $V^s(G)/T \to V^s(T)/T$. Clearly $f^*(E_T)=E_G'$ as $E_G'$ is just the restriction of $E_T$.  Considering the square
\[
\xymatrix{
  E_G'=(V^s(G) \times E)/T \ar[r] \ar[d] &E_G =(V^s(G) \times E)/G \ar[d] \\
  V^s(G)/T \ar[r]^-{\pi} & V^s(G)/G,
}
\]
we see that $\pi^*(E_G) =E_G'$. Then we have that $f^*(E_T)=\pi^*(E_G)$, and so in particular $f^*(c_i(E_T)=\pi^*(c_i(E_G))$.  The result now follows from Martin's theorem (Theorem~\ref{martin}).
\end{proof}
The corollary shows that the restriction of Martin's isomorphism to degree 2 is just
\[
i:c_1(W_i) \mapsto \sum_{j=1}^{r_i} c_1(W_{v_{ij}}).
\]
In particular we have that $i(\omega_{M_Q})=\omega_{M_{Q^{\ab}}}$.

\begin{prop}
  Let $\Amp(Q)$,~$\Amp(Q^{\ab})$ denote the ample cones of $M_Q$ and $M_{Q^{\ab}}$ respectively. Then 
  \[
  i(\Amp(Q))=\Amp(Q^{\ab})^W.
  \]
\end{prop}
\begin{proof}
Let $v$ be a character for $G$, denoting its image under $i: \chi(G) \hookrightarrow \chi(T)$ as $v$ as well. First note that  $V^{ss}_v(G) \subset V^{ss}_v(T)$.
To see this, suppose $v \in V$ is semi-stable for $G$,~$v.$ Let $\lambda: \C^* \to T$ be a one-parameter subgroup of $T$ such that $\lim_{t \to 0}\lambda(t) \cdot v$ exists. By inclusion, $\lambda$ is a one-parameter subgroup of $G$, and so $\langle v, \lambda \rangle \geq 0$ by semi-stability of $v$. Hence $v \in V^{ss}_v(T).$   It follows that, if $v \in \NE^1(M_Q)$, then $V^{ss}_v(G) \neq \emptyset$, so $V^{ss}_v(T) \neq \emptyset$, and hence $v \in \NE^1(M_{Q^{\ab}})^W.$ 

Ciocan-Fontanine--Kim--Sabbah use duality to construct a projection~\cite{CiocanFontanineKimSabbah2008} 
\[
p: \NE_1(M_{Q^{\ab}}) \to \NE_1(M_Q).
\]
Suppose that $\alpha \in \Amp(Q)$. Then for any $C \in \NE_1(M_{Q^{\ab}})$, $i(\alpha) \cdot C=\alpha \cdot p(C) \geq 0.$
So $i(\alpha) \in \Amp(Q^{\ab})^W$.


Let $\Wall(G) \subset \Pic(M_Q)$ denote the union of all GIT walls given by the $G$ action, and similarly for $\Wall(T)$. Recall that $v \in \Wall(G)$ if and only if it has a non-empty strictly semi-stable locus. Suppose $v \in \Wall(G)$, with $v$ in the strictly semi-stable locus. That is, there exists a non-trivial $\lambda: \C^* \to G$ such that $\lim_{t \to 0} \lambda(t) \cdot v$ exists and $\langle v,\lambda \rangle=0$. Now we don't necessarily have $\image(\lambda)\subset T$, but the image is in some maximal torus, and hence there exists $g \in G$ such that $\image(\lambda) \subset g^{-1} T g.$ Consider $\lambda'=g \lambda g^{-1}$. Then $\lambda'(\C^*) \subset T.$ Since $g\cdot v$ is in the orbit of $v$ under $G$, it is semi-stable with respect to $G$, and hence with respect to $T$. In fact, it is strictly semi-stable with respect to $T$, since $\lim_{t \to 0} \lambda'(t) g \cdot v=\lim_{t \to 0} g \lambda(t) \cdot v$ exists, and $\langle v, \lambda' \rangle=\langle v, \lambda \rangle=0$.
So as a character of $T$, $v$ has a non-empty strictly semi-stable locus, and we have shown that
\[
i(\Wall(G)) \subset \Wall(T)^W.
\]
This means that the boundary of $i(\Amp(Q))$ is not contained in $\Amp(Q^{\ab})^W$. Since both are full dimensional cones in the $W$ invariant subspace, the inclusion $i(\Amp(Q)) \subset \Amp(Q^{\ab})^W$ is in fact an equality. 
\end{proof}

\begin{eg} Consider again the example 
\begin{center}
  \includegraphics[scale=0.5]{plucker.pdf} 
\end{center}
The Abelianization of this quiver is
\begin{center}
  \includegraphics[scale=0.5]{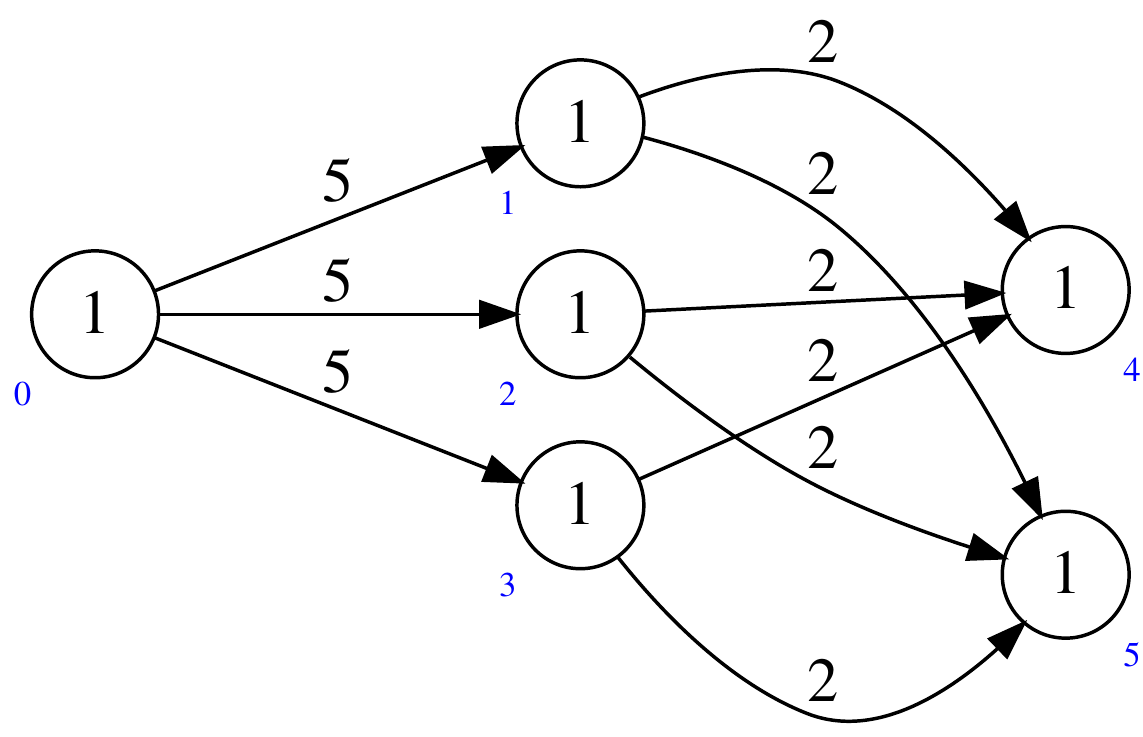} 
\end{center}
Walls are generated by collections of divisors that generate cones of codimension 1. We then intersect them with the Weyl invariant subspace, generated by $(1,1,1,0,0)$ and $(0,0,0,1,1)$. In this subspace, the walls are generated by 
\[
(1,1,1,0,0),\quad (0,0,0,1,1),\quad (-2, -2, -2, 3, 3).
\]
This is consistent with the previous example. 
\end{eg}

\subsection{The toric case}

As a prelude to determining the ample cone of a general quiver flag variety, we first consider the toric case.  Recall that a smooth projective toric variety (or orbifold) can be obtained as a GIT quotient of $\C^N$ by an $r$-dimensional torus. 

\begin{mydef} The \emph{GIT data} for a toric variety is an $r$-dimensional torus $K$ with cocharacter lattice $L=\Hom(\Cstar, K)$, and $N$ characters $D_1, \dots, D_N \in L^\vee$, together with a stability condition $w \in L^\vee \otimes \R$. 
\end{mydef} 

These linear data give a toric variety (or Deligne--Mumford stack) as the quotient of an open subset $U_w \subset \C^N$ by $K$, where $K$ acts on $\C^N$ via the map $K \rightarrow (\C^{*})^{N}$ defined by the $D_i$. $U_w$ is defined as
\[
\Big\{ (z_1,\dots,z_N) \in \C^N \; \Big|\; w \in \text{Cone}(D_i: z_i \neq 0) \Big\},
\]
that is, its elements can have zeroes at $z_i, i \in I$, only if $w$ is in the cone generated by $D_i$, $i \not \in I$. Assume that all cones given by subsets of the divisors that contain $w$ are full dimensional, as is the case for toric quiver flag varieties. Then the ample cone is the intersection of all of these.

In \cite{CrawSmith}, the GIT data for a quiver flag variety is detailed; we present it slightly differently. The torus is $K=(\C^*)^\rho$. Let $e_1,\ldots,e_\rho$ be standard basis of $L^\vee = \Z^\rho$ and $e_0=0.$ Then each $a \in Q_1$ gives a weight $D_a=-e_{s(a)}+e_{t(a)}.$ The stability condition is $\one = (1,1,\ldots,1)$. Identify $L^\vee \cong \Pic M(Q,\one).$ Then $D_a=W_a := W_{s(a)}^* \otimes W_{t(a)}.$
 
A minimal full dimensional cone  for a toric quiver flag variety is given by $\rho$ linearly independent $D_{a_i}$,~$a_i \in Q_1$. Therefore for each vertex $i$ with $1 \leq i \leq \rho$, we need an arrow $a_i$ with either $s(a)=i$ or $t(a)=i$, and these arrows should be distinct. For the positive span of these divisors to contain $\one$ requires that $D_{a_i}$ has $t(a_i)=i.$ Fix such a set $S = \{a_1, \dots, a_\rho\}$, and denote the corresponding cone by $C_S$. As mentioned, the ample cone is the intersection of such cones $C_S$. The set $S$ determines a path from $0$ to $i$ for each $i$, given by concatenating (backwards) $a_i$ with $a_{s(a_i)}$ and so on; let us write $f_{ij}=1$ if $a_j$ is in the path from $0$ to $i$, and $0$ otherwise. Then 
\[
e_i =\sum_{j=1}^\rho f_{ij} D_{a_j}.
\]
This gives us a straightforward way to compute the cone $C_S$. Let $B_S$ be the matrix with columns given by the $D_{a_i}$, and let $A_S=B_S^{-1}$. The columns of $A_S$ are given by the aforementioned paths: the $j$th column of $A_S$ is $\sum_{i=1}^\rho f_{ij} e_i$. If $c \in \Amp(Q)$, then $A_S c \in A_S \Amp(Q) \subset A_S C_S$. Since $A_S D_{a_i}=e_i$, this means that $A_S c$ is in the positive orthant. 

\begin{prop} Let $M(Q,\one)$ be a toric quiver flag variety. Let $c \in \Amp(Q)$, $c=(c_1,\dots,c_\rho)$, be an ample class, and suppose that vertex $i$ of the quiver $Q$ satisfies the following condition: for all $j \in Q_0$ such that $j>i$, there is a path from $0$ to $j$ not passing through $i$. Then $c_i>0$.
\end{prop}
\begin{proof} Choose a collection $S$ of arrows $a_j \in Q_1$ such that the span of the associated divisors $D_{a_j}$ contains the stability condition $\one$, and such that the associated path from $0$ to $j$ for any $j>i$ does not pass through $i$. Then the $(i,i)$ entry of $A_S$ is $1$ and all other entries of the $i^{th}$ row are zero. As $A_S c$ is in the positive orthant, $c_i>0.$
\end{proof}
\begin{cor}
Let $M(Q,\br)$ be a quiver flag variety, not necessarily toric. If $c=(c_1,\dots,c_\rho) \in \Amp(Q)$ and $r_j>1$, then $c_j > 0$.
\end{cor}
\begin{proof}
Consider the Abelianized quiver. For any vertex $v \in Q^{\ab}_0$, we can always choose a path from the origin to $v$ that does not pass through $v_{j1}$: if there is an arrow between $v_{j1}$ and $v$, then there is an arrow between $v_{j2}$ and $v$, so any path through $v_{j1}$ can be rerouted through $v_{j2}$. Then we obtain that the $j1$ entry of $i(c)$ is positive -- but this is just $c_j.$\end{proof}

\subsection{The ample cone of a quiver flag variety}

Let $M(Q,\br)$ be a quiver flag variety and $Q'$ be the associated Abelianized quiver. For each $i \in \{1,\dots,\rho\}$, define
\[
T_i:=\{j \in Q_0 \mid \text{all paths from $0$ to $j$ pass through $i$}\},
\]
but pretending that a path can always circumvent a vertex $j$ when $r_j>1$. This is motivated by the path structure of the Abelianized quiver.  Note that $i \in T_i$, and that if $r_i=1$ then $T_i=\{i\}$. 

\begin{thm} \label{thm:nef cone}
The nef cone of $M(Q,\br)$ is given by the following inequalities. Suppose that $a=(a_1,\dots,a_\rho) \in \Pic(M_Q)$. Then $a$ is nef if and only if
\begin{align} \label{eq:nef cone}
  \sum_{j \in T_i} r_j a_j \geq 0 &&i = 1,2,\ldots,\rho.
\end{align}
\end{thm}
\begin{proof}
We have already shown that the Weyl invariant part of the nef cone of $M_{Q'}:=M(Q',\one)$ is the image of the nef cone of $M_Q:=M(Q,\br)$ under the natural map $\pi: \Pic(M_Q) \to \Pic(M_{Q'})$. Label the vertices of $Q'$ as $v_{ij}$, $i \in \{0,\dots,\rho\}$, $j \in \{1, \dots,r_i\},$ and index elements of $\Pic(M_{Q'})$ as $(b_{ij}).$ The inequalities defining the ample cone of $M_{Q'}$ are given by a choice of arrow $A_{ij} \in Q'_1, t(A_{ij})=v_{ij}$ for each $v_{ij}$.  This determines a path $P_{ij}$ from $0 \to v_{ij}$ for each vertex~$v_{ij}$. For each $v_{ij}$ the associated inequality is:
\begin{equation}\label{ine} \sum_{v_{ij} \in P_{kl}} b_{kl} \geq 0.
\end{equation}

Suppose that $a$ is nef. First we show that $a$ satisfies the inequalities \eqref{eq:nef cone}. We have shown that $a_j \geq 0$ if $r_j>1$, so it suffices to consider $i$ such that $r_i=1$. For each such $i$, take a choice of arrows such that if $v_{kl} \in P_{i1}, k \in T_i.$ Moreover, choose arrows such that $s(A_{ik})= s(A_{il})$ for all $k,l$. Then this gives the inequalities \eqref{eq:nef cone}, after restricting to the Weyl invariant locus. Therefore, if $C$ is the cone defined by \eqref{eq:nef cone}, we have shown that $\Nef(M_Q) \subset C$. 

Suppose now that $a \in C$ and take a choice of arrows $A_{kl}$. Write $\pi(a)=(a_{ij})$. We prove that the inequalities \ref{ine}  are satisfied starting at $v_{\rho \rho}$. For $\rho$, the inequality is $a_{\rho \rho} \geq 0$, which is certainly satisfied. Suppose the $(i j+1),(i j+2),\dots,(\rho \rho)$ inequalities are satisfied. The inequality we want to establish for $(ij)$ is 
\[
a_{i}+\sum_{k \in T_i-\{i\}} r_k a_{k} + X=a_{ij}+\sum_{k \in T_i-\{i\}} \sum_{l=1}^{r_l} a_{kl} + \Gamma \geq 0,
\]
where
\[
\Gamma=\sum_{s(A_{kl})=v_{ij}, k \not \in T_i}\left(a_{kl}+\sum_{v_{kl} \in P_{st}} a_{st}\right).
\]
As $a \in C$ it suffices to show that $\Gamma \geq 0.$ By the induction hypothesis $a_{kl}+\sum_{v_{kl} \in P_{st}} a_{st} \geq 0$,
and therefore $\Gamma \geq 0$. This shows that $\pi(a)$ satisfies \eqref{ine}. 
\end{proof}

\subsection{Nef line bundles are globally generated}

We conclude this section by proving that nef line bundles on quiver flag varieties are globally generated. This is well-known for toric varieties.  This result will be important for us because in order to use the Abelian/non-Abelian Correspondence to compute the quantum periods of quiver flag zero loci, we need to know that the bundles involved are convex.  Convexity is a difficult condition to understand geometrically, but it is implied by global generation.

Let $M(Q,\br)$ be a quiver flag variety and $Q'$ be the associated Abelianized quiver.   For each $i \in \{1,\dots,\rho\}$, let $T_i$ be as above, and define:
\begin{itemize}[topsep=0.1ex, itemsep=0pt,parsep=0pt]
\item $H_i:=\{j \in Q_0 \mid \text{$j \ne 0$ and all paths from $0$ to $i$ pass through $j$}\}$;
\item $h(i):=\max{H_i}; h'(i):=\min{H_i}$; and
\item $R_i:=\{j \mid h'(j)=i\}$;
\end{itemize}
again pretending that a path can always circumvent a vertex $j$ when $r_j>1.$ Note that $i \in H_i$.  Observe too that the $R_i$ give a partition of the vertices of $Q$, and that if $j$,~$k \in T_i$, then $j$,~$k \in R_{h'(i)}$.
\begin{prop} \label{pro:nef}
Let $L$ be a nef line bundle on $M(Q,r)$. Then $L$ is globally generated.
\end{prop}
\begin{proof}

Let $L$ be a nef line bundle on $M(Q,\br)$ given by $(a_1,\dots,a_\rho)$. A section of $L$ is a $G$-equivariant section of the trivial line bundle on $\Rep(Q,\br)$, where the action of $G$ on the line bundle is given by the character $\prod \chi_i^{a_i}.$ A point of $\Rep(Q,\br)$ is given by $(\phi_a)_{a \in Q_1}, \phi_a:\C^{r_{s(a)}} \to \C^{r_{t(a)}},$ where $G$ acts by change of basis. A choice of path $i \to j$ on the quiver gives an equivariant map $\Rep(Q,\br) \to \Hom(\C^{r_i},\C^{r_j})$ where $G$ acts on the image  by $g \cdot \phi = g_j \phi g_i^{-1}.$ If $r_i=r_j=1$, such maps can be composed. 

We will show that $L$ is globally generated. It suffices to show this in the case that $\{j \mid a_j \neq 0\} \subset R_i$ for some $i$. This is because if $j \in R_i$, the inequalities in Theorem~\ref{thm:nef cone} which involve $j$ only involve $a_k$ for $k \in R_i$. So if $L=L_1 \otimes \cdots \otimes L_\rho$ such that $L_i$ only has non-zero powers of $\det(W_j)$, $j \in R_i$, then $L$ is nef if and only if all the $L_i$ are nef. So suppose that $\{j \mid a_j \neq 0\} \subset R_i.$ If $r_i>1,$ then $R_i=\{i\}$ and $L=\det(W_i)^{\otimes a_i}$, which is known to be globally generated. So we further assume that $r_i=1$.

Let $f_i$ be any homogenous polynomial of degree $d_i=\sum_{k \in T_i} r_k a_k \geq 0$ in the maps given by paths $0 \to i$. That is, $f_i$ is a $G$-equivariant map $\Rep(Q,\br) \to \C$, where $G$ acts on the image with character $\chi_i^{d_i}$.  For $j$ in $R_i$ with $r_j=1$ and $j \ne i$, let $f_j$ be any homogenous polynomial of degree $\sum_{k \in T_j} r_k a_k\geq 0$ in the maps given by paths $h(j) \to j$ (note that by definition, $r_{h(j)}=1$). That is, $f_j$ is a $G$-equivariant map $\Rep(Q,\br) \to \C$, where $G$ acts on the image with character $\chi_{h(j)}^{-d_j} \chi_j^{d_j}$.  For $j$ in $R_i$ with $r_j>1$, let $f_j$ be a homogenous polynomial of degree $a_k \geq 0$ in the minors of the matrix whose columns are given by the paths $h(j) \to j$. That is, $f_j$ is a $G$-equivariant map $\Rep(Q,\br) \to \C$, where $G$ acts on the image with character $\chi_{h(j)}^{-r_j a_j} \chi_j^{a_j}$. For any $x \in \Rep(Q,\br)$ which is semi-stable, and for any $j \in R_i$, there exists such an $f_j$ with $f_j(x) \neq 0$, by construction of $h(j)$. Define $s':=\prod_{j \in R_i} f_j :\Rep(Q,\br) \to \C$, and take the section $s$ to be a sum of such $s'.$ Then $s$ is a $G$-equivariant map $\Rep(Q,\br) \to \C$, where $G$ acts on the image with character 
$$\prod_{j \in R_i} \chi_j^{b_j}=\chi_i^{d_i} \cdot \prod_{j \in R_i, j \neq i, r_j=1} \chi_{h(j)}^{-d_j} \chi_j^{d_j} \cdot \prod_{j \in R_i, j \neq i, r_j>1} \chi_{h(j)}^{-r_j a_j} \chi_j^{a_j}.$$

We need to check that $b_j=a_j$ for all $j$. This is obvious for $j \in R_i$ with  $r_j>1.$ For $j=i$, 
\[
b_i=\sum_{j \in T_i} r_j a_j - \sum_{k \in R_i, k \neq i, h(k)=i} \sum_{j \in T_k} r_k a_k.
\]
This simplifies to $a_i$ in view of the fact that for all $j \in T_i$, there is a unique $k$ such that $h(k)=i$ and $j \in T_k$. The check for $j \in R_i$ with $r_j=1$ is similar. Therefore $s$ gives a well-defined section of $L$. For any $x \in \Rep(Q,\br)$ semi-stable, there exists an $s$ such that $s(x) \neq 0,$ so $L$ is globally generated. 

\end {proof}

\section{The Abelian/non-Abelian Correspondence}

The main theoretical result of this paper, Theorem~\ref{thm:AnA} below, proves the Abelian/non-Abelian Correspondence with bundles~\cite[Conjecture~6.1.1]{CiocanFontanineKimSabbah2008} for quiver flag zero loci.  This determines all genus-zero Gromov--Witten invariants, and hence the quantum cohomology, of quiver flag varieties, as well as all genus-zero Gromov--Witten invariants of quiver flag zero loci involving cohomology classes that come from the ambient space.  In particular, it determines the \emph{quantum period} of a quiver flag varieties or quiver flag zero locus $X$ with $c_1(TX) \geq 0$.

\subsection{A brief review of Gromov--Witten theory} \label{sec:review}
We give a very brief review of Gromov--Witten theory, mainly to fix notation,  See \cite{CoatesCortiGalkinKasprzyk2016,CiocanFontanineKimSabbah2008} for more details and references.  Let $Y$ be a smooth projective variety. Given $n \in \Z_{\geq 0}$ and $\beta \in H_2(Y)$, let  $M_{0,n}(Y,\beta)$ be the moduli space of genus zero stable maps to $Y$ of class $\beta$, and with $n$ marked points~\cite{Kontsevich95}. While this space may be highly singular and have components of different dimensions, it has a \emph{virtual fundamental class} $[M_{0,n}(Y,\beta)]^{virt}$ of the expected dimension~\cite{BehrendFantechi97,LiTian98}. There are natural evaluation maps $ev_i: M_{0,n}(Y,\beta) \to Y$ taking a class of a stable map $f: C \to Y$ to $f(x_i)$, where $x_i \in C$ is the $i^{th}$ marked point. There is also a line bundle $L_i \to M_{0,n}(Y,\beta)$ whose fiber at $f: C \to Y$ is the cotangent space to $C$ at $x_i$. The first Chern class of this line bundle is denoted $\psi_i$. Define:
\begin{equation}
  \label{eq:GW}
  \langle \tau_{a_1}(\alpha_1),\dots,\tau_{a_n}(\alpha_n) \rangle_{n,\beta} = \int_{[M_{0,n}(Y,\beta)]^{virt}} \prod_{i=1}^n ev_i^*(\alpha_i) \psi_i^{a_i}
\end{equation}
where the integral on the right-hand side denotes cap product with the virtual fundamental class.  If $a_i=0$ for all $i$, this is called a (genus zero) Gromov--Witten invariant and the $\tau$ notation is omitted; otherwise it is called a descendant invariant. It is deformation invariant. 

We consider a generating function for  descendant invariants called the \emph{J-function}. 
Write $q^{\beta}$ for the element of $\Q[H_2(Y)]$ representing $\beta \in H_2(Y)$. Write $N(Y)$ for the Novikov ring of $Y$:
\[
N(Y) = \left\{\sum_{\beta \in \NE_1(Y)} c_\beta q^\beta \: \middle| \: \begin{minipage}{3in}
$c_\beta \in \C$, for each $d \geq 0$ there are only finitely many $\beta$ such that $\omega \cdot \beta \leq d$ and $c_\beta \ne 0$ 
\end{minipage}
\right\}.
\]
Here $\omega$ is the K\"ahler class on $Y$.  The J-function assigns an element of $H^*(Y) \otimes N(Y)[[z^{-1}]]$ to every element of $H^*(Y)$, as follows. Let $\phi_1,\dots,\phi_N$ be a homogenous basis of $H^*(Y)$, and let $\phi^1,\dots,\phi^N$ be the Poincar\'e dual basis.  Then the J-function maps
\begin{equation}
  \label{eq:J}
  T \in H^*(Y) \mapsto 1+Tz^{-1}+z^{-1} \sum_i \langle\langle \phi_i/(z-\psi)\rangle\rangle \phi^i.
\end{equation}
Here $1$ is the unit class in $H^0(Y)$, and
\begin{equation}
  \label{eq:J_correlator}
  \langle\langle \phi_i/(z-\psi)\rangle\rangle=\sum_{\beta \in \NE_1(Y)} q^\beta \sum_{n=0}^\infty \sum_{a=0}^\infty\frac{1}{n! z^{a+1}} \langle \tau_a(\phi_i),T,\dots,T\rangle_{n+1,\beta}.
\end{equation}
The \emph{small} J-function is the restriction of the J-function to $H^0(Y)\oplus H^2(Y)$; closed forms for the small J-function of toric complete intersections and toric varieties are known~\cite{Givental98}. The quantum period $G_Y(t)$ is the component of $J(0)$ along $1 \in H^\bullet(Y)$ after the substitutions $z \mapsto 1$, $q^\beta \mapsto t^{\langle -K_Y,\beta \rangle}.$ This is a power series in $t$. The quantum period satisfies an important differential equation called the quantum differential equation. 

Let $X \subset Y$ be the zero locus of a generic section of a convex vector bundle $E \to Y$ and let $\be$ denote the total Chern class, which evaluates on a vector bundle $F$ of rank $r$ as 
\[
\be(F) = \lambda^r + \lambda^{r-1} c_1(F) + \cdots + \lambda c_{r-1}(F) + c_r(F).
\]
The notation here indicates that one can consider $\be(F)$ as the $\C^*$-equivariant Euler class of $F$, with respect to the canonical action of $\C^*$ on $F$ which is trivial on the base of $F$ and scales all fibers.  In this interpretation, $\lambda \in H^\bullet_{\C^*}(pt)$ is the equivariant parameter.  The twisted J-function $J_{\be,E}$ is defined exactly as the J-function \eqref{eq:J}, but replacing the virtual fundamental class which occurs there (via equations~\ref{eq:J_correlator} and~\ref{eq:GW}) by $[M_{0,n}(Y,\beta)]^{virt} \cap \be(E_{0,n,\beta})$, where $E_{0,n,\beta}$ is $\pi_*(ev_{n+1}^*(E))$, $\pi: M_{0,n+1}(Y,\beta) \to M_{0,n}(Y,\beta)$ is the universal curve, and $ev_{n+1}:M_{0,n}(Y,\beta) \to Y$ is the evaluation map. $E_{0,n,\beta}$ is a vector bundle over $M_{0,n}(Y,\beta)$, because $E$ is convex.  Functoriality for the virtual fundamental class~\cite{KimKreschPantev2003} implies that
\[
j^* J_{\be,E}(T)\big|_{\lambda=0} = J_X(j^*T)
\]
where $j:X \to Y$ is the embedding~\cite[Theorem~1.1]{Coates2014}.  Thus one can compute the quantum period of $X$ from the twisted J-function. We will use this to compute the quantum period of Fano fourfolds which are quiver flag zero loci. 

The Abelian/non-Abelian correspondence is a conjecture \cite{CiocanFontanineKimSabbah2008} relating the J-functions (and more broadly, the quantum cohomology Frobenius manifolds) of GIT quotients $V/\!\!/G$ and $V/\!\!/T$, where $T \subset G$ is the maximal torus. It also extends to considering zero loci of representation theoretic bundles, by relating the associated twisted J-functions. As the Abelianization $V/\!\!/T$ of a quiver flag variety $V/\!\!/G$ is a toric quiver flag variety, the Abelian/non-Abelian correspondence conjectures a closed form for the J-functions of Fano quiver flag zero loci. Ciocan-Fontanine--Kim--Sabbah proved the Abelian/non-Abelian correspondence (with bundles) when $V/\!\!/G$ is a flag manifold~\cite{CiocanFontanineKimSabbah2008}. We will build on this to prove the conjectures when $V/\!\!/G$ is a quiver flag variety.  

\subsection{The I-Function}

We give the J-function in the way usual in the literature: first, by defining a cohomology-valued hypergeometric function called the I-function (which should be understood as a mirror object, but we omit this perspective here), then relating the J-function to the I-function. We follow the construction given by \cite{CiocanFontanineKimSabbah2008} in our special case.  Let $X$ be a quiver flag zero locus given by $(Q,E_G)$, and write $M_Q=M(Q,\br)$ for the ambient quiver flag variety.  Let $(Q^{\ab}, E_T)$ be the associated Abelianized quiver and bundle, $M_{Q^{\ab}}=M(Q^{\ab},(1,\dots,1))$. Assume, moreover, that $E_T$ splits into nef line bundles; this implies that both $E_T$ and $E_G$ are convex. To define the I-function, we need to relate the Novikov rings of $M_Q$ and $M_{Q^{\ab}}$. Let $\Pic Q$ (respectively $\Pic Q^{\ab}$) denote the Picard group of $M_Q$ (respectively of $M_{Q^{\ab}}$), and similarly for the cones of effective curves and effective divisors. The isomorphism $\Pic Q  \to (\Pic Q^{\ab})^W$ gives a projection $p:\NE_1(M_{Q^{\ab}}) \to \NE_1(M_Q).$ In the bases dual to $\det(W_1),\dots,\det(W_\rho)$ of $\Pic M_Q$ and $W_{ij}, 1 \leq i \leq \rho, 1 \leq j \leq r_i$ of $\Pic M_{Q^{\ab}}$, this is
\[
p \colon (d_{1,1},\dots,d_{1,r_1},d_{2,1},\dots,d_{\rho,r_\rho}) \mapsto (\sum_{i=1}^{r_1} d_{1 i},\dots,\sum_{i=1}^{r_\rho} d_{\rho i}).
\]
For $\beta=(d_1,\dots,d_\rho)$, define
\[
\epsilon(\beta)=\sum_{i=1}^\rho {d_i(r_i-1)}.
\]
Then, following \cite[equation~3.2.1]{CiocanFontanineKimSabbah2008}, the induced map of Novikov rings $N(M_{Q^{\ab}}) \to N(M_Q)$ sends
\[
q^{\tilde{\beta}} \mapsto  (-1)^{\epsilon(\beta)} q^{\beta}
\]
where $\beta = p(\tilde{\beta})$.  We write $\tilde{\beta} \to \beta$ if and only if $\tilde{\beta} \in \NE_1(M_{Q^{\ab}})$ and $p(\tilde{\beta}) = \beta$.

For a representation theoretic bundle $E_G$ of rank $r$ on $M_Q$, let $D_1,\dots, D_r$ be the divisors on $M_{Q^{\ab}}$ giving the split bundle $E_T$. Given $\tilde{d} \in \NE_1(M_{Q^{\ab}})$ define 
\[
I_{E_G}(\tilde{d})=\frac{\prod_{i=1}^r \prod_{m \leq \langle \tilde{d}, D_i \rangle} (D_i+m z)}{\prod_{i=1}^r \prod_{m \leq 0} (D_i+m z)}.
\]
Notice that all but finitely many factors cancel here.  If $E$ is K-theoretically a representation theoretic bundle, in the sense that there exists $A_G, B_G$ such that
\[
0 \to A_G \to B_G \to E \to 0
\]
is an exact sequence, we define
\begin{equation}
  \label{eq:K_theory}
  I_E(\tilde{d})=\frac{I_{B_G(\tilde{d})}}{I_{A_G(\tilde{d})}}.
\end{equation}
      
\begin{eg}
The Euler sequence from Proposition \ref{euler} shows that for the tangent bundle $T_{M_Q}$
\[
I_{T_{M_Q}}(\tilde{d})=\frac{\prod_{a \in Q_1^{ab}}\prod_{m \leq 0} (D_a+m z)}{\prod_{a \in Q_1^{ab}}\prod_{m \leq \langle \tilde{d}, D_a \rangle} (D_a+m z)}\frac{\prod_{i=1}^\rho \prod_{j \neq k} \prod_{m \leq \langle \tilde{d}, D_{ij}-D_{ik} \rangle} (D_{ij}-D_{ik}+m z)}{\prod_{i=1}^\rho \prod_{j \neq k} \prod_{m \leq 0} (D_{ij}-D_{ik}+m z)}.
\]
Here $D_a$ is the divisor on $M_{Q^{ab}}$ corresponding to the arrow $a \in Q^{ab}_1$, and $D_{ij}$ is the divisor corresponding to the tautological bundle $W_{ij}$ for vertex $ij$.
\end{eg}

\begin{eg}
If $X$ is a quiver flag zero locus in $M_Q$ defined by the bundle $E_G$, then the adjunction formula implies that
\[
I_{T_X}(\tilde{d})=I_{T_{M_Q}}(\tilde{d}) I_{E_G}(\tilde{d}).
\]
\end{eg}

Define the I-function of $X \subset M_Q$ to be
\[
I_{X,M_Q} =\sum_{d \in \NE_1(M_Q)} \sum_{\tilde{d} \to d} (-1)^{\epsilon(d)} q^d I_{T_X}(\tilde{d}).
\]
Since
\begin{equation}\label{TX} 
I_{T_X}(\tilde{d})=\frac{\prod_{a \in Q_1^{ab}}\prod_{m \leq 0} (D_a+m z)}{\prod_{a \in Q_1^{ab}}\prod_{m \leq \langle \tilde{d}, D_a \rangle} (D_a+m z)}\frac{\prod_{i=1}^\rho \prod_{j \neq k} \prod_{m \leq \langle \tilde{d}, D_{ij}-D_{ik} \rangle} (D_{ij}-D_{ik}+m z)}{\prod_{i=1}^\rho \prod_{j \neq k} \prod_{m \leq 0} (D_{ij}-D_{ik}+m z)}
\end{equation}
is invariant under the action of the Weyl group on the $D_{ij}$, by viewing these as Chern roots of the tautological bundles $W_i$ we can express it as a function in the Chern classes of the $W_i$. We can therefore regard $I_{T_X}(\tilde{d})$ as an element of $H^\bullet(M_Q,\C)$.  Thus the I-function is an element of $H^\bullet(M_Q,\C) \otimes N(M_Q) \otimes \C(\!(z^{-1})\!)$.  If $X$ is Fano then 
\begin{equation}
  \label{eq:I_asymptotics}
  I_{X,M_Q}(z)=1+z^{-1} C+O(z^{-2})
\end{equation}
where $O(z^{-2})$ denotes terms of the form $\alpha z^{k}$ with $k \leq {-2}$ and $C \in  H^0(M_Q,\C) \otimes N(M_Q)$; furthermore $C$ vanishes if the Fano index of $X$ is greater than $1$.

\begin{thm} \label{thm:AnA} 
  Let $X$ be a Fano quiver flag zero locus given by $(Q,E_G)$, and let $j \colon X \to M_Q$ be the embedding of $X$ into the ambient quiver flag variety.  Then
  \[
  J_X(0) = e^{-c/z} j^* I_{X,M_Q}
  \]
  where $c = j^* C$.
\end{thm}

\begin{rem}
  Via the Divisor Equation and the String Equation~\cite[\S1.2]{Pandharipande98}, Theorem~\ref{thm:AnA} determines $J_X(\tau)$ for $\tau \in H^0(X) \oplus H^2(X)$.
\end{rem}
\subsection{Proof of Theorem~\ref{thm:AnA}}

Givental has defined~\cite{Givental04,CoatesGivental07} a Lagrangian cone $\mathcal{L}_X$ in the symplectic vector space $H_X := H^*(X,\C) \otimes N(X) \otimes \C(\!(z^{-1})\!)$ that encodes all genus-zero Gromov--Witten invariants of $X$.  Note that $J_X(T) \in H_X$ for all $T$.  The J-function has the property that $(-z) J_X(T,-z)$ is the unique element of $\mathcal{L}_X$ of the form
\[
{-z} + T + O(z^{-1})
\]
and this, together with the expression \eqref{eq:I_asymptotics} for the I-function and the String Equation
\[
J_X(T+c,z) = e^{c/z} J_X(T,z)
\]
shows that Theorem~\ref{thm:AnA} follows immediately from Theorem~\ref{thm:AnA_enhanced} below.  Theorem~\ref{thm:AnA_enhanced} is stronger: it does not require the hypothesis that the quiver flag zero locus $X$ be Fano.

\begin{thm} \label{thm:AnA_enhanced}
  Let $X$ be a quiver flag zero locus given by $(Q,E_G)$, and let $j \colon X \to M_Q$ be the embedding of $X$ into the ambient quiver flag variety.  Then $(-z) j^* I_{X,M_Q}(-z) \in \mathcal{L}_X$.
\end{thm}

\begin{proof}

Let $Y=\prod_{i=1}^\rho \Gr(\tilde{s_i},r_i)$, and denote by $Y^{ab}=\prod_{i=1}^\rho (\mathbb{P}^{\tilde{s_i}-1})^{r_i}$ its Abelianization.  In \S\ref{zero loci} we constructed a vector bundle $V$ on $Y$ such that $M_Q$ is cut out of $Y$ by a regular section of $V$:
\[
V=\bigoplus_{i=2}^\rho Q_i \otimes \C^{\tilde{s}_i *}/F_i^*
\]
where $\tilde{s}_i$ is the number of paths from $0$ to $i$ and $F_i =\bigoplus_{t(a)=i} Q_{s(a)}$.  $V$ is globally generated and hence convex. It is not representation theoretic, but it is K-theoretically: the sequence
\[
0 \to F_i^* \otimes Q_i \to (\C^{\tilde{s_i}})^* \otimes Q_i \to (\C^{\tilde{s_i}})^* \otimes Q_i/F_i^* \to 0
\]
is exact.  Let $i \colon M_Q \to Y$ denote the inclusion.

Both $Y$ and $M_Q$ are GIT quotients by the same group; we can therefore canonically identify a representation theoretic vector bundle $E'_G$ on $Y$ such that $E'_G|_{M_Q}$ is $E_G$.  Our quiver flag zero locus $X$ is cut out of $Y$ by a regular section of $V' = V \oplus E'_G$. Note that
$$I_{T_{M_Q}}(\tilde{d}) I_{V}(\tilde{d}) = I_{T_Y}(\tilde{d}) I_{V'}(\tilde{d}).$$ The I-function $I_{X,M_Q}$ defined by considering $X$ as a quiver flag zero locus in $M_Q$ with the bundle $E_G$ then coincides with the pullback $i^* I_{X,Y}$ of the I-function defined by considering $X$ as a quiver flag zero locus in $Y$ with the bundle $V'$. 
 It therefore suffices to prove that
\[
(-z) (i \circ j)^* I_{X,Y}(-z) \in \mathcal{L}_X.
\]

We consider a $\C^*$-equivariant counterpart of the I-function, defined as follows.  For a representation theoretic bundle $W_G$ on $Y$, let $D_1,\dots, D_r$ be the divisors on $Y^{ab}$ giving the split bundle $W_T$, and for $\tilde{d} \in \NE_1(Y^{\ab})$ set 
\[
I^{\C^*}_{W_G}(\tilde{d})=\frac{\prod_{i=1}^r \prod_{m \leq \langle \tilde{d}, D_i \rangle} (\lambda + D_i+m z)}{\prod_{i=1}^r \prod_{m \leq 0} (\lambda + D_i+m z)}.
\]
We extend this definition to bundles on $Y$ -- such as $V'$ -- that are only K-theoretically representation theoretic in the same way as \eqref{eq:K_theory}.  Recalling that
\[
I_{T_Y}(\tilde{d})=\frac{\prod_{i=1}^\rho \prod_{j \neq k} \prod_{m \leq \langle \tilde{d}, D_{ij}-D_{ik} \rangle} (D_{ij}-D_{ik}+m z)}{\prod_{i=1}^\rho \prod_{j \neq k} \prod_{m \leq 0} (D_{ij}-D_{ik}+m z)}
\frac{\prod_{i=1}^\rho \prod_{j=1}^{r_i} \prod_{m \leq 0} (D_{ij}+m z)^{\tilde{s}_i}}{\prod_{i=1}^\rho \prod_{j=1}^{r_i}\prod_{m \leq \langle \tilde{d}, D_{ij} \rangle} (D_{ij}+m z)^{\tilde{s}_i}},
\]
we define
\[
I^{\C^*}_{X,Y} (z) = \sum_{d \in \NE_1(Y)} \sum_{\tilde{d} \to d} (-1)^{\epsilon(d)} q^d I_{T_Y}(\tilde{d}) \, I^{\C^*}_{V'}(\tilde{d}).
\]

The I-function $I_{X,Y}$ can be obtained by setting $\lambda=0$ in $I^{\C^*}_{X,Y}$.  In view of \cite[Theorem~1.1]{Coates2014}, it therefore suffices to prove that 
\[
(-z) I^{\C^*}_{X,Y} (-z) \in \mathcal{L}_{\be, V'}
\]
where $\mathcal{L}_{\be, V'}$ is the Givental cone for the Gromov--Witten theory of $Y$ twisted by the total Chern class $\be$ and the bundle $V'$.

If $V'$ were a representation theoretic bundle, this would follow immediately from the work of Ciocan-Fontanine--Kim--Sabbah: see the proof of Theorem~6.1.2 in \cite{CiocanFontanineKimSabbah2008}.  In fact $V'$ is only K-theoretically representation theoretic, but their argument can be adjusted almost without change to this situation. Suppose that $A_G$ and $B_G$ are homogenous vector bundles,
and that
\[
0 \to A_G \to B_G \to V \to 0
\]
is exact.  Then we can also consider an exact sequence
\[
0 \to A_T \to B_T \to F \to 0
\]
on the Abelianization, and define $V_T:=F.$
Using the notation of the proof of \cite[Theorem~6.1.2]{CiocanFontanineKimSabbah2008}, the point is that 
\[
\triangle(V) \triangle (A_G)=\triangle(B_G)
\]
Here, $\triangle(V)$ is the twisting operator that appears in the Quantum Lefschetz theorem~\cite{CoatesGivental07}. We can then follow the same argument for 
\[
\triangle(B_G)/\triangle(A_G)
\]
After Abelianizing, we obtain $\triangle(B_T)/\triangle(A_T)=\triangle(F)$, and conclude that
\[
(-z) I^{\C^*}_{X,Y} (-z) \in \mathcal{L}_{\be, V'}
\]
as claimed.  This completes the proof.
\end{proof} 

\newpage

\appendix

\section{Computations} \label{sec:search}
\begin{center}
\small{\sc{T. Coates, E. Kalashnikov, A. Kasprzyk}}
\end{center}
In this appendix, we describe the computer search for four dimensional Fano quiver flag zero loci with codimension at most four. Code to perform this and similar analyses, using the computational algebra system Magma~\cite{Magma}, is available at the repository~\cite{codebase}.  A database of Fano quiver flag varieties, which was produced as part of the calculation, is available at the repository~\cite{database}.

\subsection{Classifying Quiver Flag Varieties} \label{sec:classify_ambients}

The first step is to find all Fano quiver flag varieties of dimension at most 8.  A non-negative integer matrix $A=[a_{i,j}]_{0\leq i,j \leq \rho}$ and a dimension vector $\br \in \Z^{\rho+1}_{> 0}$ determine a vertex-labelled directed multi-graph: the $\rho+1$ vertices are labelled by the $r_i$, and the adjacency matrix for the graph is $A$. Such a graph, if it is acyclic with a unique source, and the label of the source is $1$, also determines a quiver flag variety. Two $(A,r)$ pairs can determine the same  graph and hence the same quiver flag varieties. 

\begin{mydef} A pair $(A,r)$ determining a quiver flag variety is in \emph{normal form} if $r$ is increasing and, under all permutations of the $\rho+1$ indices that preserve $r$, the columns of $A$ are lex minimal. 
\end{mydef}

\noindent Two pairs in normal form determine the same quiver flag variety (and hence the same graph) if and only if they are equal. 

Recall that quiver flag varieties are towers of Grassmannians (see \S\ref{sec: tower of Grassmannian bundles}), and that the $i$th step in the tower is given by the relative Grassmannian $\Gr(\mathcal{F}_i, r_i)$, where $\mathcal{F}_i$ is a vector bundle of rank~$s_i.$ Using this construction it is easy to see that if $s_i=r_i$ then this quiver flag variety is equivalent to the quiver flag variety $\tilde{Q}$ with vertex $i$ removed, and one arrow $k \to j$ for every path of the form $k \to i \to j$.  Therefore we can assume that $s_i>r_i$, and hence that every vertex contributes strictly positively to the dimension of the quiver flag variety.  With this constraint, there are only finitely many quiver flag varieties with dimension at most 8, and each such has at most 9 vertices.

The algorithm to build all quiver flag varieties with dimension at most 8 is as follows. Start with the set $S$ of all Grassmannians of dimension at most 8. Given an element of $S$ of dimension less than 8, add one extra labelled vertex and extra arrows into this vertex, in all possible ways such that the dimension of the resulting quiver flag variety is at most $8$. Put these in normal form and include them in $S$.  Repeat until there are no remaining elements of $S$ of dimension less than~$8$.  

In this way we obtain all quiver flag varieties of dimension at most 8. We then compute the ample cone and anti-canonical bundle for each, and discard any which are not Fano. We find 223044 Fano quiver flag varieties of dimension at most $8$; 223017 of dimension $4 \leq d \leq 8$.  Of these 50617 (respectively 50612) are non-toric quiver flag varieties.  

\begin{table}[h]
\label{tab:qfv}
\begin{tabular}{lcccccccc}
  \toprule
  & \multicolumn{8}{c}{$r$}  \\ \cmidrule(l){2-9}
  \multicolumn{1}{c}{$d$} & 1 & 2 & 3 & 4 & 5 & 6 & 7 & 8 \\
  \midrule
1 & 1 &   &   &   &   &   &   &   \\
2 & 2 & 3 &   &   &   &   &   &   \\
3 & 2 & 8 & 11 &   &   &   &   &   \\
4 & 3 & 17 & 44 & 48 &   &   &   &   \\
5 & 2 & 27 & 118 & 262 & 231 &   &   &   \\
6 & 4 & 41 & 264 & 903 & 1647 & 1202 &   &   \\
7 & 2 & 54 & 498 & 2484 & 7005 & 10618 & 6541 &   \\
8 & 4 & 74 & 872 & 5852 & 23268 & 54478 & 69574 & 36880 \\ \bottomrule \\
\end{tabular}
\caption{The number of Fano quiver flag varieties by dimension $d$ and Picard rank $r$}
\end{table}

\begin{rem} 
  In our codebase we define new Magma intrinsics \intrinsic{QuiverFlagVariety(A,r)}, which creates a quiver flag variety from an adjacency matrix $A$ and dimension vector $r$, and
  \begin{center}
    \intrinsic{QuiverFlagVarietyFano(id)},
  \end{center}
  which creates a Fano quiver flag variety, in normal form, from its ID~\cite{codebase, database}.  We assign IDs to Fano quiver flag varieties of dimension at most 8, in the range $\{1\ldots 223044\}$, by placing them in normal form and then ordering them first by dimension, then by Picard rank, then lexicographically by dimension vector, then lexicographically by the columns of the adjacency matrix.  We also define Magma intrinsics \intrinsic{NefCone(Q)}, \intrinsic{MoriCone(Q)}, \intrinsic{PicardLattice(Q)}, and \intrinsic{CanonicalClass(Q)} that compute the nef cone, Mori cone, Picard lattice and canonical class of a quiver flag variety $Q$, and an intrinsic \intrinsic{PeriodSequence(Q,l)} that computes the first $l+1$ terms of the Taylor expansion of the regularised quantum period of $Q$.  See \S\ref{sec:computations} for more details.
\end{rem}

\subsection{The Class of Vector Bundles that We Consider} \label{sec:which_bundles}

We consider all bundles $E$ on a given quiver flag variety that:
\begin{itemize}
\item are direct sums of bundles of the form 
  \begin{equation}
    \label{eq:summand}
    L \otimes S^{\alpha_1}(W_1) \otimes \cdots \otimes S^{\alpha_\rho}(W_\rho)
  \end{equation}
  where each $S^{\alpha_i}$ is a positive Schur power and $L$ is a nef line bundle; and
\item have rank $c$, where $c$ is four less than the dimension of the ambient quiver flag variety.
\end{itemize}
Remark~\ref{gg} shows that positive Schur powers $S^\alpha(W_i)$ are globally generated, and Proposition~\ref{pro:nef} shows that nef line bundles are globally generated.  Since the tensor product of globally generated vector bundles is globally generated, the first condition ensures that $E$ is globally generated.  In particular, therefore, the zero locus $X$ of a generic section of $E$ is smooth.  The second condition ensures that the zero locus $X$, if non-empty, is a fourfold.  Global generation also implies that the bundle $E$ is convex, which allows us to compute the quantum period of $X$ as described in \S\ref{sec:review}.

Consider a summand as in \eqref{eq:summand}.  We can represent the partition $\alpha_i$ as a length $r_i$ decreasing sequence of non-negative integers, and write $L=\bigotimes_{j=1}^\rho (\det W_j)^{a_j}$ where $a_j$ may be negative. Therefore each such summand is determined a length $\rho$ sequence of \emph{generalised partitions}:  the partition (with possibly negative entries) corresponding to index $i$ is $\alpha_i + (a_i,\dots,a_i).$

\begin{rem}
In our codebase we define a new Magma intrinsic 
\begin{center}
  \intrinsic{QuiverFlagBundle(Q,[A1,\ldots,Ak])}
\end{center}
which creates a bundle of the above form, on the quiver flag variety $Q$, from a sequence of generalised partitions $(A1,\ldots,Ak)$.  We also define an intrinsic \intrinsic{FirstChernClass(E)} that computes the first Chern class of such a bundle $E$; intrinsics \intrinsic{Degree(E)} and \intrinsic{EulerNumber(E)} that compute the degree and Euler number\footnote{This is the Euler characteristic of $X$ as a topological space.} of the zero locus $X$ of a generic section of $E$; and intrinsics \intrinsic{HilbertCoefficients(E,l)} and \intrinsic{PeriodSequence(E,l)} that compute the first $l+1$ terms of, respectively, the Hilbert series of $X$ and the Taylor expansion of the regularised quantum period of~$X$.  See \S\ref{sec:computations}.
\end{rem}

\subsection{Classifying Quiver Flag Bundles} \label{sec:classify_bundles}
In this step, we describe the algorithm for determining all bundles on a given quiver flag variety that determine a smooth four-dimensional Fano quiver flag zero locus.
A vector bundle as above is determined by a tuple $(A, r, P)$, where $A$ is an adjacency matrix, $r$ is a dimension vector $r$, and $P=(P_1,\dots,P_k)$ is a sequence where each $P_i$ is a length-$\rho$ sequence of generalised partitions such that the $j$th partition in each $P_i$ is of length $r_j$.  Note that we regard the summands \eqref{eq:summand} in our vector bundles as unordered; also, as discussed above, different pairs $(A,r)$ can determine the same quiver flag variety.  We therefore say that a tuple $(A,r,P)$ is in \emph{normal form} if the pair $(A,r)$ is in normal form, $P$ is in lex order, and under all permutations of the vertices preserving these conditions, the sequence $P$ is lex minimal; we work throughout with tuples in normal form.

Given a Fano quiver flag variety $M(Q,\br)$ of dimension $4+c$, $c \leq 4$, with anti-canonical class $-K_Q$ and nef cone $\Nef(Q)$, we search for all bundles $E$ such that 
\begin{itemize}
\item $E$ is a direct sum of bundles of the form \eqref{eq:summand};
\item $\rk(E)=c$;
\item $-K_Q-c_1(E) \in \Amp(Q)$.
\end{itemize}
The last condition ensures that the associated quiver flag zero locus $X$, if non-empty, is Fano.  We proceed as follows.  We first find all possible summands that can occur; that is, all irreducible vector bundles $E$ of the form \eqref{eq:summand} such that $\rk(E) \leq c$ and  $-K_Q-c_1(E) \in \Amp(Q)$.
Let $\Irr(Q)$ be the set of all such bundles. Write $\Irr(Q)=\Irr(Q)_1 \sqcup \Irr(Q)_2$, where $\Irr(Q)_1$ contains vector bundles of rank strictly larger than $1$, and $\Irr(Q)_2$ contains only line bundles. We then search for two vector bundles $E_1$,~$E_2$ such that $E_i$ is a direct sum of bundles from $\Irr(Q)_i$ and that $E=E_1\oplus E_2$ satisfies the conditions above. 

For each $x \in \Nef(Q)$ such that $-K_Q-x$ is ample, we find all possible ways to write $x$ as
\begin{equation}
  \label{eq:decomposition}
  x=\sum_{i=1}^l a_i
\end{equation}
where the $a_i$ are (possibly repeated) elements of a Hilbert basis for $\Nef(Q)$. There are only finitely many decompositions \eqref{eq:decomposition}; finding them efficiently is a knapsack-type problem that has already been solved~\cite{CoatesKasprzykPrince2015}. For each $\tilde{c} \leq c$ and each partition of the $a_i$ into at most $c/2$ groups $S_1,\dots, S_{s}$, we find all possible choices of $F_1,\dots,F_{s} \in \Irr_1$ such that 
\begin{align*}
  c_1(F_i)=\sum_{j \in S_i} a_j && \rk(F_1)+\cdots+\rk(F_s)=\tilde{c}.
\end{align*}
Set $E_1=F_1 \oplus \cdots \oplus F_s$. Then for each $y \in \Nef(Q)$ such that $-K_Q-x-y$ is ample, we again find all ways of writing 
\[
y=\sum_{j=1}^m b_j
\]
as a sum of Hilbert basis elements. Each partition of the $b_j$ into $c-\tilde{c}$ groups gives a choice of nef line bundles $L_1,\dots,L_{c-\tilde{c}} \in \Irr_2(Q)$, and we set $E_2=\oplus L_j$. 

\begin{rem}
  Treating the higher rank summands $\Irr_2$ and line bundles $\Irr_1$ separately here is not logically necessary, but it makes a huge practical difference to the speed of the search.
\end{rem}

\subsection{Classifying Quiver Flag Zero Loci}

For each of the Fano quiver flag varieties $Q$ of dimension between $4$ and $8$, found in \S\ref{sec:classify_ambients}, we use the algorithm described in \S\ref{sec:classify_bundles} to find all bundles on $Q$ of the form described in \S\ref{sec:which_bundles}.
This produces 10788446 bundles.  Each such bundle $E$ determines a quiver flag zero locus $X$ that is either empty or a smooth Fano fourfold.  We discard any varieties that are empty or disconnected, and for the remainder compute the first fifteen terms of the Taylor expansion of the regularised quantum period of $X$, using Theorem~\ref{thm:AnA}.  For many of the quiver flag zero loci that we find this computation is extremely expensive, so in practice it is essential to use the equivalences described in \S\ref{sec:equivalences} to replace such  quiver flag zero loci by equivalent and more tractable models.  The number of equivalence classes is far smaller than the number of quiver flag zero loci that we found, and so this replaces roughly 10 million calculations, many of which are hard, by around half a million calculations, almost all of which are easy. In this way we find 749 period sequences.  We record these period sequences, together with the construction, Euler number, and degree for a representative quiver flag zero locus, in Appendix~\ref{results} below. \numberofnewFanos~of the period sequences that we find are new. Thus we find at least 141 new four-dimensional Fano manifolds\footnote{To be precise: we find at least 141 four-dimensional Fano manifolds for which the regularised quantum period was not previously known.  The regularised quantum period of a Fano manifold $X$ is expected to completely determine $X$.  See~\cite{CoatesKasprzykPrince2015,CoatesGalkinKasprzykStrangeway2018} for known quantum periods.}.

\begin{rem}
  A computationally cheap sufficient condition for a quiver flag zero locus to be empty arises as follows.   If $W$ is the tautological quotient bundle on $Gr(n,r)$, where $2r-1>n$, then a generic global section of $\wedge^2 W$ or $\Sym^2 W$ has an empty zero locus. Thus if $i$ is a vertex in a quiver $Q$ such that all arrows into $i$ are from the source, and $2r_i-1>n_{0i}=s_i$, then there are no global sections of $\wedge^2 W_i$ or $\Sym^2 W_i$ with non-empty zero locus: to see this, apply Proposition~\ref{prop:zl} to $Q$.
\end{rem}

\subsection{Cohomological Computations for Quiver Flag Zero Loci} \label{sec:computations}

In this section we describe how we compute the degree, Euler characteristic, Hilbert series, and Taylor expansion of the regularised quantum period for quiver flag varieties and quiver flag zero loci. This relies on Martin's integration formula~\cite{Martin2000} and Theorem \ref{thm:AnA}. 

Let $V$ be a smooth projective variety with an action of $G$ on $V$, let $T$ be a maximal torus in $G$, and consider the GIT quotients $V/\!\!/G$ and $V/\!\!/T$ determined by a character of $G$. Let $\pi:V^{ss}(G)/T \to V/\!\!/G$ be the projection and $i:V^{ss}(G)/T \to V^{ss}(T)/T=V/\!\!/T$ be the inclusion. Let $W$ be the Weyl group, and $e=\prod_{\lambda \in {Roots}(G)} c_1(L_\lambda),$ where $L_\lambda$ is the line bundle on $V/\!\!/T$ associated to the character $\lambda$.
\begin{thm}[Martin's Integration Formula, \cite{Martin2000}] For any $a \in H^*(V/\!\!/G, \C)$ and any $\tilde{a} \in H^*(V/\!\!/T, \C)$ satisfying $\pi^*(a)=i^*(\tilde{a})$
\[
\int_{V/\!\!/G} a = \frac{1}{|W|}\int_{V/\!\!/T} \tilde{a} \cup e.
\]
\end{thm}
\noindent If $a \in H^*(V/\!\!/G, \C)$ and $\tilde{a} \in H^*(V/\!\!/T, \C)$ satisfy $\pi^*(a)=i^*(\tilde{a})$ then we say that $\tilde{a}$ is a lift of $a$.

In our case the Abelianization $V/\!\!/T$ is a smooth toric variety, and the cohomology rings of such varieties, being Stanley--Reisner rings, are easy to work with computationally~\cite{Magma,Sage}.  For example, we can use this to compute the number of components $h^0(X,\mathcal{O}_X)$ of a Fano quiver flag zero locus $X$.  By Kodaira vanishing, $h^0(X,\mathcal{O}_X)=\chi(X)$, and applying the Hirzebruch--Riemann--Roch theorem gives
\begin{equation}
  \label{eq:HRR}
  \chi(\mathcal{O}_X)=\int_X ch(\mathcal{O}_X) \cup Td(T_X)=\int_X Td(T_X).
\end{equation}
We need to find a lift of the Todd class of $T_X$. Writing $T_X$ as a K-theoretic quotient of representation theoretic bundles via the Euler sequence, as in the proof of Theorem \ref{thm:AnA},  gives the lift that we seek; we then use Martin's formula to reduce the integral \eqref{eq:HRR} to an integral in the cohomology ring of the Abelianization.  The same approach allows us to compute the first two terms $\chi(X,{-K_X})$, $\chi(X,{-2K_X})$ of the Hilbert series of $X$ -- which determine the entire Hilbert series, since $X$ is a Fano fourfold -- as well as the degree and Euler characteristic of $X$.  To compute the first few Taylor coefficients of the quantum period of $X$, we combine this approach with the explicit formula in Theorem~\ref{thm:AnA}.

\newpage

\section{Regularized Quantum Periods for Quiver Flag Zero Loci}
\begin{center}
\small{\sc{T. Coates, E. Kalashnikov, A. Kasprzyk}}
\end{center}
\label{results}

\subsection{The Table of Representatives}

As described in Appendix~\ref{sec:search}, we divided the 4-dimensional quiver flag zero loci $X$ that we found into 749 buckets, according to the first 15 terms of the Taylor expansion of the regularised quantum period of $X$.  We refer to these Taylor coefficients as the period sequence.  Table~\ref{tab:descriptions} below gives, for each of the 749 period sequence buckets, a representative quiver flag zero locus $X$ as well as the degree and Euler number of $X$.  (In some cases we do not know that all the quiver flag zero loci in a bucket are isomorphic, but we checked that they all have the same degree, Euler number, and Hilbert series.)  The quiver flag zero locus $X$ is represented by the adjacency matrix and dimension vector of its ambient quiver flag variety $Y = M(Q,\br)$, together with the sequence of generalised partitions that determine a vector bundle $E \to Y$ such that $X$ is the zero locus of a generic section of $E$.  The generalised partitions are written as Young diagrams, with:
\begin{itemize}
\item $\varnothing$ representing the empty Young diagram;
\item a filled Young diagram, such as $\Ylinecolour{white}\Yfillcolour{black}\yng(2)\Yfillcolour{white}\Ylinecolour{black}$, representing the dual to the vector bundle represented by the unfilled Young diagram $\yng(2)$.  
\end{itemize}
Filled Young diagrams that occur always represent line bundles.

The entries in Table~\ref{tab:descriptions} give representatives of each period sequence bucket that are chosen so as to make the computation of geometric data (the period sequence etc.) straightforward\footnote{They are chosen to minimize the quantity $\sum_{i=0}^\rho r_i^2$, which is a rough proxy for the complexity of the Chow ring of the Abelianization.}.  Even though the Table is constructed by considering all four-dimensional Fano manifolds that occur as quiver flag zero loci in codimension up to four, in a four cases there is no tractable representative as a quiver flag zero locus of low codimension.  In these cases the Table contains a representative as a quiver flag zero locus in higher codimension; the reader who prefers models in lower-dimensional ambient spaces should consult Table~\ref{smaller_ambient}.

\begin{table}[h]
  \centering
  \begin{tabular}{cccc}
    \toprule
    Period ID & Adjacency matrix & Dimension vector & Generalized partitions \\ \midrule
    \oddrow \hypertarget{desc:73}{}\hyperlink{ps:73}{$73$}&
$\cimatrix{$\begin{array}{ccccc}
0 & 1 & 1 & 3 & 3 \\
0 & 0 & 1 & 0 & 0 \\
0 & 0 & 0 & 0 & 0 \\
0 & 0 & 0 & 0 & 0 \\
0 & 1 & 0 & 0 & 0 \\\end{array}$}$
&$\cimatrix{$\begin{array}{ccccc}1 & 2 & 2 & 2 & 2 \\\end{array}$}$
& $\left(\emptyset, \yng(1), \yng(1), \emptyset\right)$\\
\evnrow \hypertarget{desc:144}{}\hyperlink{ps:144}{$144$}&
$\cimatrix{$\begin{array}{ccccc}
0 & 3 & 1 & 2 & 3 \\
0 & 0 & 0 & 1 & 0 \\
0 & 0 & 0 & 0 & 0 \\
0 & 0 & 0 & 0 & 0 \\
0 & 0 & 1 & 0 & 0 \\\end{array}$}$
&$\cimatrix{$\begin{array}{ccccc}1 & 1 & 2 & 2 & 2 \\ \end{array}$}$
&$\left( \emptyset, \yng(1), \yng(1), \emptyset\right)$\\
\oddrow \hypertarget{desc:439}{}\hyperlink{ps:439}{$439$}&
$\cimatrix{$\begin{array}{ccc}
0 & 1 & 5 \\
0 & 0 & 0 \\
0 & 1 & 0 \\\end{array}$}$
&$\cimatrix{$\begin{array}{ccc}1 & 4 & 4 \\\end{array}$}$
&$\left(\yng(1,1,1,1), \emptyset\right), \left(\yng(1,1,1,1), \emptyset\right), \left(\yng(1,1,1,1), \emptyset\right), \left(\yng(1,1,1,1), \emptyset\right)$\\
\evnrow \hypertarget{desc:552}{}\hyperlink{ps:552}{$552$}&
$\cimatrix{$\begin{array}{ccc}
0 & 0 & 5\\
0 & 0 & 0\\
0 & 1 & 0\\\end{array}$}$
&$\cimatrix{$\begin{array}{ccc}1 & 2 & 4 \\\end{array}$}$
&$\left(\yng(1), \yng(1,1,1,1)\right), \left(\yng(1,1), \emptyset\right), \left(\yng(1,1),\emptyset\right)$\\ 
\bottomrule \\
  \end{tabular}
  \caption{Representatives for certain Period IDs in codimension at most four}
  \label{smaller_ambient}
\end{table}

\begin{rem}
  The data in Tables~\ref{tab:descriptions} and~\ref{qfv_periods} can also be found, in machine readable form, in the ancillary files that accompany this paper.
\end{rem}

\subsection{The Table of Period Sequences}

Table~\ref{qfv_periods} records the first 8 terms of the period sequence, $\alpha_0,\alpha_1,\ldots,\alpha_7$, for each of the 749 period sequence buckets.  It also records, where they exist, the names of known four-dimensional Fano manifolds which have the same first fifteen terms of the period sequence.  Notation is as follows:
\begin{itemize}
\item $\PP^n$ denotes $n$-dimensional complex projective space;
\item $Q^n$ denotes a quadric hypersurface in $\PP^{n+1}$;
\item $\FI{4}{k}$ is the $k$th four-dimensional Fano manifold of index~$3$, as in \cite[\S5]{CoatesGalkinKasprzykStrangeway2018};
\item $\VV{4}{k}$ is the $k$th four-dimensional Fano manifold of index~$2$ and Picard rank 1, as in \cite[\S6.1]{CoatesGalkinKasprzykStrangeway2018};
\item $\MW{4}{k}$ is the $k$th four-dimensional Fano manifold of index~$2$ and Picard rank at least 2, as in \cite[\S6.2]{CoatesGalkinKasprzykStrangeway2018};
\item $\Obro{4}{k}$ is the $k$th four-dimensional toric Fano manifold, as in \cite[\S7]{CoatesGalkinKasprzykStrangeway2018};
\item $\Str_{k}$ are the Strangeway fourfolds described in \cite[\S8]{CoatesGalkinKasprzykStrangeway2018};
\item $\CKP_{k}$ is the $k$th four-dimensional toric complete intersection, as in~\cite{CoatesKasprzykPrince2015};
\item $\SS_k$ denotes the del~Pezzo surface of degree~$k$;
\item $\VV{3}{k}$ denotes the three-dimensional Fano manifold of Picard rank~$1$, Fano index~$1$, and degree~$k$;
\item $\BB{3}{k}$ denotes the three-dimensional Fano manifold of Picard rank~$1$, Fano index~$2$, and degree~$8k$;
\item $\MM{\rho}{k}$ denotes the $k$th entry in the Mori--Mukai list of three-dimensional Fano manifolds of Picard rank~$\rho$~\cite{MoriMukai1982,MoriMukai1983,MoriMukai1986,MoriMukai2003,MoriMukai2004}.  We use the ordering as in~\cite{CoatesCortiGalkinKasprzyk2016}, which agrees with the original papers of Mori--Mukai except when $\rho=4$.
\end{itemize}

\begin{rem}
It appears from Table~\ref{qfv_periods} as if the period sequences with IDs $72$ and $73$ might coincide. This is not the case. The coefficients $\alpha_8$,~$\alpha_9$, and~$\alpha_{10}$ in these cases are:\\
\begin{center}
\small
\setlength{\extrarowheight}{0.2em}
\begin{tabular}{cccc}
\toprule
Period ID  &$\alpha_{8}$&$\alpha_{9}$&$\alpha_{10}$\\
\midrule
\oddrow \hyperlink{ps:72}{$72$}&$32830$&$212520$&$1190952$\\
\evnrow \hyperlink{ps:73}{$73$}&$32830$&$227640$&$1190952$\\
\bottomrule \\
\end{tabular}
\end{center}
\end{rem}

\begin{rem}
  590 of the period sequences that we find coincide with period sequences for toric complete intersections, at least for the first 15 terms.  579 of these are realised by quiver flag zero loci that are also toric complete intersections.  For the remaining 11 cases -- period sequences with IDs 17, 48, 73, 144, 145, 158, 191, 204, 256, 280, and 282 -- there is no model as a toric complete intersection that is also a quiver flag zero locus in codimension at most four.  In four of these cases -- with IDs 17, 48, 144, and 256 -- the toric complete intersection period sequence is realised by a smooth four-dimensional toric variety.
\end{rem}

\begin{rem}
  An earlier version of this paper omitted two of the period sequences that we find below, due to erroneous hand calculations in special cases.  In this version all computations are performed in software, in a uniform way; we believe that this makes them more likely to be correct.
\end{rem}

\subsection*{Acknowledgements}

EK was supported by the Natural Sciences and Engineering Research Council of Canada, and by the EPSRC Centre for Doctoral Training in Geometry and Number Theory at the Interface, grant number EP/L015234/1.  TC was supported by ERC Consolidator Grant number~682602 and EPSRC Programme Grant EP/N03189X/1.  AK was supported by EPSRC Fellowship grant EP/N022513/1.  The computations that underpin this work were performed on the Imperial College HPC cluster.  We thank Andy Thomas, Matt Harvey, and the Research Computing Service team at Imperial for invaluable technical assistance.

\renewcommand{\baselinestretch}{1}
\setlength{\LTcapwidth}{24cm}
\setlength{\extrarowheight}{0.2em}
\newgeometry{left=1cm,right=1cm,top=2.45cm,bottom=2.45cm}
\begin{landscape}
\begin{longtable}{clllcc}
\caption{Certain $4$-dimensional Fano manifolds with Fano index $1$ that arise as quiver flag zero loci}
\label{tab:descriptions}\\
\toprule
\multicolumn{1}{c}{Period ID}&\multicolumn{1}{c}{Adjacency matrix}&\multicolumn{1}{c}{Dimension vector}&\multicolumn{1}{c}{Generalized partitions}&\multicolumn{1}{c}{Degree}&\multicolumn{1}{c}{Euler Number}\\
\midrule
\endfirsthead
\multicolumn{6}{l}{\tiny Continued from previous page.}\\
\addlinespace[1.7ex]
\midrule
\multicolumn{1}{c}{Period ID}&\multicolumn{1}{c}{Adjacency matrix}&\multicolumn{1}{c}{Dimension vector}&\multicolumn{1}{c}{Generalized partitions}&\multicolumn{1}{c}{Degree}&\multicolumn{1}{c}{Euler Number}\\
\midrule
\endhead
\midrule
\multicolumn{6}{r}{\tiny Continued on next page.}\\
\endfoot
\bottomrule
\endlastfoot
\oddrow \hypertarget{desc:1}{}\hyperlink{ps:1}{$1$}&$\cimatrix{$\begin{array}{cc}
0&5\\0&0\\\end{array}$}$&$\cimatrix{$\begin{array}{cc}
1&1\\\end{array}$}$&&$625$&5\\
\evnrow \hypertarget{desc:2}{}\hyperlink{ps:2}{$2$}&$\cimatrix{$\begin{array}{ccc}
0&0&5\\0&0&0\\0&3&0\\\end{array}$}$&$\cimatrix{$\begin{array}{ccc}
1&1&1\\\end{array}$}$&$\left(\yng(1),\Ylinecolour{white}\Yfillcolour{black}\yng(1)\Yfillcolour{white}\Ylinecolour{black}\right)$,$\left(\yng(1),\varnothing\right)$&$512$&8\\
\oddrow \hypertarget{desc:3}{}\hyperlink{ps:3}{$3$}&$\cimatrix{$\begin{array}{cc}
0&6\\0&0\\\end{array}$}$&$\cimatrix{$\begin{array}{cc}
1&1\\\end{array}$}$&$\left(\yng(2)\right)$&$512$&6\\
\evnrow \hypertarget{desc:4}{}\hyperlink{ps:4}{$4$}&$\cimatrix{$\begin{array}{ccc}
0&1&5\\0&0&0\\0&1&0\\\end{array}$}$&$\cimatrix{$\begin{array}{ccc}
1&1&1\\\end{array}$}$&$\left(\yng(1),\yng(1)\right)$&$431$&9\\
\oddrow \hypertarget{desc:5}{}\hyperlink{ps:5}{$5$}&$\cimatrix{$\begin{array}{ccc}
0&2&3\\0&0&0\\0&1&0\\\end{array}$}$&$\cimatrix{$\begin{array}{ccc}
1&1&1\\\end{array}$}$&&$513$&9\\
\evnrow \hypertarget{desc:6}{}\hyperlink{ps:6}{$6$}&$\cimatrix{$\begin{array}{ccc}
0&1&3\\0&0&0\\0&2&0\\\end{array}$}$&$\cimatrix{$\begin{array}{ccc}
1&1&1\\\end{array}$}$&&$513$&9\\
\oddrow \hypertarget{desc:7}{}\hyperlink{ps:7}{$7$}&$\cimatrix{$\begin{array}{cccc}
0&1&2&2\\0&0&1&0\\0&0&0&0\\0&1&0&0\\\end{array}$}$&$\cimatrix{$\begin{array}{cccc}
1&1&1&1\\\end{array}$}$&&$459$&12\\
\evnrow \hypertarget{desc:8}{}\hyperlink{ps:8}{$8$}&$\cimatrix{$\begin{array}{cccc}
0&0&2&3\\0&0&0&0\\0&1&0&0\\0&1&1&0\\\end{array}$}$&$\cimatrix{$\begin{array}{cccc}
1&1&1&1\\\end{array}$}$&$\left(\yng(1),\varnothing,\varnothing\right)$&$417$&13\\
\end{longtable}
\newpage
\setlength{\LTcapwidth}{\textwidth}
\begin{longtable}{cp{4.5cm}cccccccc}
\caption{Some regularized period sequences obtained from $4$-dimensional Fano manifolds that arise as quiver flag zero loci.}
\label{qfv_periods}\\
\toprule
\multicolumn{1}{c}{Period ID}&\multicolumn{1}{c}{Name}&\multicolumn{1}{c}{$\alpha_{0}$}&\multicolumn{1}{c}{$\alpha_{1}$}&\multicolumn{1}{c}{$\alpha_{2}$}&\multicolumn{1}{c}{$\alpha_{3}$}&\multicolumn{1}{c}{$\alpha_{4}$}&\multicolumn{1}{c}{$\alpha_{5}$}&\multicolumn{1}{c}{$\alpha_{6}$}&\multicolumn{1}{c}{$\alpha_{7}$}\\
\midrule
\endfirsthead
\multicolumn{10}{l}{\tiny Continued from previous page.}\\
\addlinespace[1.7ex]
\midrule
\multicolumn{1}{c}{Period ID}&\multicolumn{1}{c}{Name}&\multicolumn{1}{c}{$\alpha_{0}$}&\multicolumn{1}{c}{$\alpha_{1}$}&\multicolumn{1}{c}{$\alpha_{2}$}&\multicolumn{1}{c}{$\alpha_{3}$}&\multicolumn{1}{c}{$\alpha_{4}$}&\multicolumn{1}{c}{$\alpha_{5}$}&\multicolumn{1}{c}{$\alpha_{6}$}&\multicolumn{1}{c}{$\alpha_{7}$}\\
\midrule
\endhead
\midrule
\multicolumn{10}{r}{\tiny Continued on next page.}\\
\endfoot
\bottomrule
\endlastfoot
\oddrow \hypertarget{ps:1}{}\hyperlink{desc:1}{$1$}&$\Obro{4}{124}$, $\PP^{4}$, $\CKP_{1}$&$1$&$0$&$0$&$0$&$0$&$120$&$0$&$0$\\
\evnrow \hypertarget{ps:2}{}\hyperlink{desc:2}{$2$}&$\Obro{4}{115}$, $\CKP_{2}$&$1$&$0$&$0$&$0$&$24$&$120$&$0$&$0$\\
\oddrow \hypertarget{ps:3}{}\hyperlink{desc:3}{$3$}&$\CKP_{3}$, $Q^{4}$&$1$&$0$&$0$&$0$&$48$&$0$&$0$&$0$\\
\evnrow \hypertarget{ps:4}{}\hyperlink{desc:4}{$4$}&$\CKP_{4}$&$1$&$0$&$0$&$0$&$48$&$120$&$0$&$0$\\
\oddrow \hypertarget{ps:5}{}\hyperlink{desc:5}{$5$}&$\Obro{4}{118}$, $\CKP_{8}$&$1$&$0$&$0$&$6$&$0$&$120$&$90$&$0$\\
\evnrow \hypertarget{ps:6}{}\hyperlink{desc:6}{$6$}&$\Obro{4}{47}$, $\CKP_{10}$&$1$&$0$&$0$&$6$&$24$&$0$&$90$&$2520$\\
\oddrow \hypertarget{ps:7}{}\hyperlink{desc:7}{$7$}&$\Obro{4}{94}$, $\CKP_{11}$&$1$&$0$&$0$&$6$&$24$&$120$&$90$&$1260$\\
\evnrow \hypertarget{ps:8}{}\hyperlink{desc:8}{$8$}&$\Obro{4}{37}$, $\CKP_{12}$&$1$&$0$&$0$&$6$&$24$&$120$&$90$&$2520$\\
\end{longtable}
\end{landscape}
\newgeometry{margin=3cm}

\bibliographystyle{habbrv}
\bibliography{bibliography}

\end{document}